\documentclass[12pt, a4paper, reqno]{amsart}
\usepackage{eurosym}
\usepackage{amsmath}
\usepackage{amsfonts}
\usepackage{amssymb}
\usepackage{color}
\usepackage{comment,graphicx}
\usepackage[T1]{fontenc}
\usepackage{url}
\usepackage{ae}
\usepackage{mathtools}
\usepackage{esint}
\usepackage{latexsym}
\usepackage{amscd}
\usepackage{amsmath}
\usepackage{amsfonts}
\usepackage{amssymb}
\usepackage{color}
\usepackage{comment,graphicx}
\usepackage[T1]{fontenc}
\usepackage{url}
\usepackage{ae}

\def\phi{\varphi}
\def\rho{\varrho}
\def\epsilon{\varepsilon}
\numberwithin{equation}{section}
\theoremstyle{plain}
\newtheorem{theorem}[equation]{Theorem}
\newtheorem{lemma}[equation]{Lemma}

\newtheorem{corollary}[equation]{Corollary}
\theoremstyle{definition}
\newtheorem{definition}[equation]{Definition}

\theoremstyle{remark}
\newtheorem{remark}[equation]{Remark}

\renewcommand{\le}{\leqslant}
\renewcommand{\ge}{\geqslant}
\renewcommand{\leq}{\leqslant}
\renewcommand{\geq}{\geqslant}

\newcommand{\vertiii}[1]{{\left\vert\kern-0.25ex\left\vert\kern-0.25ex\left\vert \right\vert\kern-0.25ex\right\vert\kern-0.25ex\right\vert}}

\begin{document}
\title[Mixed-norm Herz-type Triebel-Lizorkin spaces]{Mixed-norm
Herz-type Triebel-Lizorkin spaces}
\author[D. Drihem]{Douadi Drihem}
\address{Douadi Drihem\\
Laboratory of Functional Analysis and Geometry of Spaces, Faculty of
Mathematics and Informatics, Department of Mathematics,\\
M'sila University, PO Box 166 Ichebelia, M'sila 28000, Algeria}
\email{douadidr@yahoo.fr, douadi.drihem@univ-msila.dz}
\thanks{ }
\date{\today }

\begin{abstract}
	Based on mixed-norm Herz spaces, we introduce a new class of mixed-norm Herz-type Triebel-Lizorkin spaces. We establish their $\varphi$-transform characterization in the sense of Frazier and Jawerth and prove Sobolev embedding theorems for these spaces. The resulting embedding theorems extend and improve several known results for mixed-norm Triebel-Lizorkin spaces.

\end{abstract}

\maketitle
\section{Introduction}
Function spaces play a fundamental role in modern mathematical analysis, particularly in the study of partial differential equations, harmonic analysis, and the regularity theory of solutions. Among the most important scales are the classical Sobolev, Bessel potential, Besov, and Triebel-Lizorkin spaces. These spaces provide a unified framework for describing smoothness and integrability properties of functions and distributions. In recent decades, considerable attention has been devoted to extending these classical scales by replacing the underlying Lebesgue spaces with more general function spaces that better capture anisotropic behavior, local regularity, and different types of integrability.

One important generalization is provided by Herz spaces, which were introduced to describe the behavior of functions through weighted integrability over dyadic annuli centered at the origin. In contrast to the classical Lebesgue spaces, Herz spaces allow one to measure local integrability and global decay independently. This property makes Herz spaces particularly well suited to the study of problems in which the behavior of functions near the origin and at infinity plays a fundamental role. Herz-type Sobolev and Bessel potential spaces were investigated by Lu and Yang \cite{LuYang97}, who established their basic properties and explored their applications to partial differential equations. Subsequently, Herz-type  Triebel-Lizorkin spaces emerged as natural extensions of the classical smoothness scales, incorporating Herz norms to capture both local regularity and spatial decay; see, for instance, \cite{XuYang05}.

The theory of Herz-type Triebel-Lizorkin spaces has attracted particular interest because of its applications to nonlinear partial differential equations and harmonic analysis. In particular, the Cauchy problem for semilinear parabolic equations has been investigated in Herz-type Triebel-Lizorkin spaces under suitable assumptions on the nonlinear terms, illustrating the importance of these spaces in the analysis of nonlinear phenomena; see \cite{Dr22.Banach}.

A further generalization arises from the theory of mixed-norm function spaces, in which different integrability exponents are assigned to different coordinate directions. Mixed-norm spaces are especially useful in problems exhibiting anisotropic behavior, where the regularity or integrability of a function may vary from one direction to another. Recently, Zhang, Yang, and Zhang \cite{ZYZ22} introduced and investigated a class of mixed-norm Herz spaces, providing a natural framework that combines the annular structure of Herz spaces with the directional flexibility of mixed-norm Lebesgue spaces. They established several fundamental properties of these spaces, including duality results, a Riesz-Thorin interpolation theorem, the boundedness of the Hardy-Littlewood maximal operator, and the Fefferman-Stein vector-valued maximal inequality.

These developments suggest that mixed-norm Herz spaces provide a suitable underlying framework for constructing generalized smoothness spaces that incorporate both anisotropic integrability and global decay. In particular, the study of Triebel-Lizorkin-type spaces associated with mixed-norm Herz spaces is a natural continuation of the existing theory. Such spaces are expected to combine the advantages of mixed-norm Lebesgue spaces, Herz spaces, and classical Triebel-Lizorkin spaces, thereby providing a flexible setting for investigating embeddings, smoothness properties, and applications to partial differential equations.

The main purpose of the present work is to develop the theory of mixed-norm Herz-type  Triebel-Lizorkin spaces. These spaces are constructed by combining the Littlewood-Paley decomposition with mixed-norm Herz-type integrability. The resulting scales extend both the classical mixed-norm  Triebel-Lizorkin spaces and the corresponding Herz-type spaces. The principal objective is to establish their fundamental properties and to investigate the relationships between these spaces through characterizations and embedding theorems.

More specifically, we first recall the basic properties of mixed-norm Herz spaces and present some auxiliary estimates required for the subsequent analysis. We then obtain a 
$\phi$-transform characterization of the newly introduced spaces by means of the Calder\'{o}n reproducing formula, following the approach of Frazier and Jawerth. Furthermore, we establish elementary embedding results and Sobolev-type embeddings.

The results developed in this paper provide a systematic foundation for the study of mixed-norm Herz-type Triebel-Lizorkin spaces and clarify their connections with several established scales of function spaces. They also open the way for further investigations into applications to harmonic analysis and partial differential equations, particularly in settings where anisotropic integrability and spatial decay must be treated simultaneously.

\section{Mixed-norm  Herz spaces}

The aim of this section is to provide the necessary background on mixed-norm Herz spaces and present several technical results that will be needed in the subsequent sections of the paper, including the boundedness of the maximal operator. These results will play a crucial role in the development of the main theory. Throughout this paper, we adopt the following notation and conventions.

\subsection{Notation and conventions}
 
Throughout the paper, $\mathbb{R}^{n}$ denotes the $n$-dimensional Euclidean
space, $\mathbb{N}$ the set of positive integers, and $\mathbb{N}_{0}=%
\mathbb{N}\cup\{0\}$. The symbol $\mathbb{Z}$ stands for the set of all
integers.

For two non-negative quantities $f$ and $g$, we write $f\lesssim g$ if there
exists a positive constant $c$, independent of the relevant parameters, such
that $f\leq cg$. We write $f\approx g$ if both $f\lesssim g$ and $g\lesssim f
$ hold.

For a function $f$, $\operatorname{supp} f$ denotes its support. If
$E\subset\mathbb{R}^{n}$ is a measurable set, then $|E|$ denotes its
Lebesgue measure and $\chi_E$ its characteristic function. 

Let $E\subset\mathbb{R}^{n}$ be measurable and let $0<p\leq\infty$. We
denote by $L^{p}(E)$ the space of all measurable functions $f:E\to\mathbb{C}$
such that 
\begin{equation*}
\|f\|_{L^{p}(E)} = \left(\int_{E}|f(x)|^{p}\,dx\right)^{1/p} <\infty, \qquad
0<p<\infty, 
\end{equation*}
and 
\begin{equation*}
\big\|f\big\|_{L^{\infty }(E)}=\underset{x\in E}{\text{ess-sup}}%
|f(x)|<\infty .
\end{equation*}
When $E=\mathbb{R}^{n}$, we simply write $L^{p}$ instead of $L^{p}(\mathbb{R}%
^{n})$ and $\|f\|_{p}$ instead of $\|f\|_{L^{p}(\mathbb{R}^{n})}$.

For $\vec p=(p_{1},\ldots,p_{n})$ with $p_{i}\in(0,\infty]$, $i\in\{1,\ldots,n\}$, we write $%
0<\vec p\leq\infty$. The mixed-norm Lebesgue space $L^{\vec p}(\mathbb{R}%
^{n})$ consists of all measurable functions $f:\mathbb{R}^{n}\to\mathbb{C}$
satisfying 
\begin{equation*}
\|f\|_{L^{\vec p}(\mathbb{R}^{n})} = \Bigg(
\int_{\mathbb{R}} \Bigg(
\cdots \Bigg(
\int_{\mathbb{R}} |f(x_{1},\ldots,x_{n})|^{p_{1}} \,dx_{1} \Bigg)%
^{p_{2}/p_{1}} \cdots \Bigg)^{p_{n}/p_{n-1}} dx_{n} \Bigg)^{1/p_{n}}
<\infty, 
\end{equation*}
with the usual modification of replacing the corresponding integral by the
essential supremum whenever $p_{j}=\infty$. Equipped with this quasi-norm, $%
L^{\vec p}(\mathbb{R}^{n})$ is a quasi-Banach space and becomes a Banach
space whenever $\min\{p_{1},\ldots,p_{n}\}\geq 1$. If $p_{1}=\cdots=p_{n}=p$%
, then 
\begin{equation*}
L^{\vec p}(\mathbb{R}^{n})=L^{p}(\mathbb{R}^{n}). 
\end{equation*}
For further details on mixed-norm Lebesgue spaces, we refer to \cite{BP61}.

We denote by $\mathcal{S}(\mathbb{R}^{n})$ the Schwartz space of rapidly
decreasing smooth functions on $\mathbb{R}^{n}$. It is equipped with the
family of seminorms 
\begin{equation*}
\|\varphi\|_{\mathcal{S}_{M}} = \sup_{\substack{ \gamma\in\mathbb{N}_{0}^{n}
\\ |\gamma|\leq M}} \sup_{x\in\mathbb{R}^{n}} |\partial^{\gamma}\varphi(x)|
(1+|x|)^{n+M+|\gamma|}, \qquad M\in\mathbb{N}. 
\end{equation*}
Its dual space, denoted by $\mathcal{S}^{\prime }(\mathbb{R}^{n})$, is the
space of tempered distributions.

The Fourier transform of $f\in\mathcal{S}(\mathbb{R}^{n})$ is defined by 
\begin{equation*}
\mathcal{F}f(\xi) = (2\pi)^{-n/2} \int_{\mathbb{R}^{n}}
e^{-ix\cdot\xi}f(x)\,dx, \qquad \xi\in\mathbb{R}^{n}, 
\end{equation*}
and its inverse is denoted by $\mathcal{F}^{-1}f$. Both operators extend to $%
\mathcal{S}^{\prime }(\mathbb{R}^{n})$ in the usual way.

For $v\in\mathbb{N}_{0}$ and $m\in\mathbb{Z}^{n}$, let 
\begin{equation*}
Q_{v,m} = 2^{-v}\bigl(\lbrack 0,1)^{n}+m\bigr)
\end{equation*}
denote the dyadic cube of side length $2^{-v}$. We write $%
\chi_{v,m}=\chi_{Q_{v,m}}.$

Furthermore, for $i\in\{1,\ldots,n\}$, we define 
\begin{equation*}
Q_{v,m_i}^{\,i} = 2^{-v}\bigl(\lbrack 0,1)+m_i\bigr)
\end{equation*}
and $%
\chi_{v,m_{i}}=\chi_{Q_{v,m_{i}}^{\,i}}.$

For $0<\vec p\leq\infty$, $\vec\alpha=(\alpha_{1},\ldots,\alpha_{n})\in%
\mathbb{R}^{n}$, and $\vec\alpha_{j}=(\alpha_{1}^{j},\ldots,\alpha_{n}^{j})%
\in\mathbb{R}^{n}$, $j\in\{1,2\}$, we use the abbreviations 
\begin{equation*}
p_{-} = \min\{p_{1},\ldots,p_{n}\}, \qquad \frac{1}{\mathbf{p}} =
\sum_{i=1}^{n}\frac{1}{p_{i}}, 
\end{equation*}
and 
\begin{equation*}
\boldsymbol{\alpha}_{j}=\sum_{i=1}^{n}\alpha _{i}^{j}, \qquad \boldsymbol{\alpha}=\sum_{i=1}^{n}\alpha _{i},. 
\end{equation*}

Finally, $\mathfrak{M}(\mathbb{R}^{n})$ denotes the collection of all
measurable functions on $\mathbb{R}^{n}$. Any additional notation will be
introduced as needed.

\subsection{Definition and some basic properties}
 For $i\in \{1,\dots,n\}$ and $k_{i}\in 
\mathbb{Z}$, let 
\begin{equation*}
R_{k_{i}}=\{x_{i}\in \mathbb{R}:2^{k_{i}-1}\leq |x_{i}|<2^{k_{i}}\} \quad 
\text{and} \quad \chi_{k_{i}}=\chi_{R_{k_{i}}}.
\end{equation*}

\begin{definition}
Let $0<\vec{p},\vec{q}\leq \infty$ and $\vec{\alpha}=(\alpha_{1},\dots,%
\alpha_{n})\in \mathbb{R}^{n}$. The mixed-norm Herz space $\dot{E}_{\vec{p}%
}^{\vec{\alpha},\vec{q}}(\mathbb{R}^{n})$ is defined as the set of all
functions $f\in \mathfrak{M}(\mathbb{R}^{n})$ such that  
\begin{equation*}
\big\|f\big\|_{\dot{E}_{\vec{p}}^{\vec{\alpha},\vec{q}}(\mathbb{R}^{n})}= %
\big\|\cdot \cdot \cdot \big\|f\big\|_{\dot{K}_{p_{1}}^{\alpha
_{1},q_{1}}}\cdot \cdot \cdot \big\|_{\dot{K}_{p_{n}}^{\alpha
_{n},q_{n}}}<\infty ,
\end{equation*}
where  
\begin{equation*}
\|f\|_{\dot{K}_{p_{i}}^{\alpha_{i},q_{i}}} = \Big(\sum_{k_{i}\in \mathbb{Z}}
2^{k_{i}\alpha_{i}q_{i}} \|f\,\chi_{k_{i}}\|_{L^{{p_{i}}}(\mathbb R)}^{q_{i}}\Big)^{1/q_{i}},
\quad i\in \{1,\dots,n\}.
\end{equation*}
\end{definition}

\begin{remark}
When $p_{i}=q_{i}=p$ for all $i\in \{1,\dots,n\}$ and $\vec{\alpha}=\vec{0}=(0,\dots,%
0)$, the mixed-norm Herz spaces $\dot{E}_{\vec{p}}^{\vec{0},\vec{q}}(%
\mathbb{R}^{n})$ coincide with the Lebesgue space $L^{p}(\mathbb{R}^{n})$.

Moreover, when $p_{i}=q_{i}$ for all $i\in \{1,\dots,n\}$ and $\vec{\alpha}=%
\vec{0}$, they coincide with the mixed Lebesgue space $L^{\vec{p}}(%
\mathbb{R}^{n})$.

More details on mixed-norm Herz spaces can be found in \cite{ZYZ22}.
\end{remark}

The Hardy-Littlewood maximal operator $M$ is defined for locally integrable
functions by 
\begin{equation*}
Mf(x)=\sup_{r>0}\frac{1}{|B(x,r)|}\int_{B(x,r)}|f(y)|\,dy.
\end{equation*}

When the maximal operator is applied only in the variable $x_k$, using the
decomposition $x=(x^{\prime },x_k,x^{\prime \prime })$, we write 
\begin{equation*}
M_k f(x_1,\dots,x_n)=(Mf(x_1,\dots,x_n))(x_k).
\end{equation*}

Furthermore, for $t>0$, the iterated maximal operator $\mathcal{M}_t$ is
defined by 
\begin{equation*}
\mathcal{M}_{t}(f)(x)=\big(M_{n}\cdots M_{1}(|f|^{t})(x)\big)^{1/t}.
\end{equation*}

The proof of the main results of this paper is based on the following lemma;
see \cite{ZYZ22}.

\begin{lemma}
\label{Maximal-Inq}  Let $0<\beta \leq \infty$, $0<\vec{p},\vec{q}<\infty$,
and $\vec{\alpha}=(\alpha_{1},\dots,\alpha_{n})\in \mathbb{R}^{n}$ satisfy,
for each $i\in \{1,\dots,n\}$,  
\begin{equation*}
-\frac{1}{p_{i}}<\alpha_{i}<1-\frac{1}{p_{i}}, \quad \text{and} \quad
0<t<\min(\beta,p_{-},q_{-}).
\end{equation*}
Then there exists a constant $C>0$ such that, for any $\{f_{j}\}_{j\in 
\mathbb{Z}}\subset \dot{E}_{\vec{p}}^{\vec{\alpha},\vec{q}}(\mathbb{R}^{n})$%
,  
\begin{equation*}
\Big\|\Big(\sum_{j=-\infty}^{\infty}(\mathcal{M}_{t}(f_{j}))^{\beta}\Big)%
^{1/\beta}\Big\|_{\dot{E}_{\vec{p}}^{\vec{\alpha},\vec{q}}(\mathbb{R}^{n})}
\leq C \Big\|\Big(\sum_{j=-\infty}^{\infty}|f_{j}|^{\beta}\Big)^{1/\beta}%
\Big\|_{\dot{E}_{\vec{p}}^{\vec{\alpha},\vec{q}}(\mathbb{R}^{n})}.
\end{equation*}
\end{lemma}

We need the following lemma, which is basically a consequence of Hardy's
inequality in the sequence Lebesgue space $\ell ^{q}.$

\begin{lemma}
\label{lem:lq-inequality}  Let $0<a<1$ and $0<q\leq \infty$. Let  $%
\{\varepsilon_k\}_{k\in \mathbb{Z}}$ be a sequence of positive real numbers
such that  
\begin{equation*}
\big\|\{\varepsilon_k\}_{k\in \mathbb{Z}}\big\|_{\ell^q} = I < \infty.
\end{equation*}
Then the sequences  
\begin{equation*}
\big\{\delta_k\big\}_{k\in \mathbb{Z}}, \quad \delta_k =
\sum_{j=-\infty}^{k} a^{k-j}\varepsilon_j, \qquad \text{and} \qquad \big\{%
\eta_k\big\}_{k\in \mathbb{Z}}, \quad \eta_k = \sum_{j=k}^{\infty}
a^{j-k}\varepsilon_j,
\end{equation*}
belong to $\ell^q$, and  
\begin{equation*}
\big\|\{\delta_k\}_{k\in \mathbb{Z}}\big\|_{\ell^q} + \big\|\{\eta_k\}_{k\in 
\mathbb{Z}}\big\|_{\ell^q} \leq c\, I,
\end{equation*}
where $c>0$ depends only on $a$ and $q$.
\end{lemma}

We shall also need the following elementary fact.

\begin{lemma}
\label{Lp-estimate}  Let $0<p\leq \infty$ and $f_k \in L_{\mathrm{loc}}^{p}(%
\mathbb{R}^n)$ for $k\in \mathbb{Z}$.  Then, for any $0<\tau \leq \min(1,p)$%
,  
\begin{equation*}
\Big\|\sum_{k=-\infty}^{\infty} f_k \Big\|_p \leq \Big(\sum_{k=0}^{\infty}
\|f_k\|_p^{\tau}\Big)^{1/\tau}.
\end{equation*}
\end{lemma}

\section{The $\varphi$-transform characterization of the mixed-norm Herz-type Triebel-Lizorkin space}
In this section, we introduce the spaces
$\dot{E}_{\vec{p}}^{\vec{\alpha},\vec{q}}F_{\beta}^{s}$ and establish their
$\varphi$-transform characterizations.
Let $\Phi,\varphi\in\mathcal{S}(\mathbb{R}^{n})$ be such that

\begin{equation}
\operatorname{supp}\mathcal{F}\Phi
\subset \{\xi\in\mathbb{R}^{n}:|\xi|\le 2\},
\qquad
|\mathcal{F}\Phi(\xi)|\ge c>0,
\label{Ass1}
\end{equation}
for all $|\xi|\le \frac{5}{3}$, and
\begin{equation}
\operatorname{supp}\mathcal{F}\varphi
\subset
\left\{\xi\in\mathbb{R}^{n}:\frac12\le |\xi|\le 2\right\},
\qquad
|\mathcal{F}\varphi(\xi)|\ge c>0,
\label{Ass2}
\end{equation}
for all $\frac35\le |\xi|\le \frac53$, where $c$ is a positive constant.
Throughout this section, we set
\[
\widetilde{\varphi}(x)=\overline{\varphi(-x)},
\qquad x\in\mathbb{R}^{n}.
\]

We now introduce the function spaces that will be studied in this section.

\begin{definition}
	\label{B-F-def}
	Let $s\in\mathbb{R}$,
	$0<\vec{p},\vec{q}<\infty$,
	$\vec{\alpha}=(\alpha_{1},\ldots,\alpha_{n})\in\mathbb{R}^{n}$,
	$0<\beta\le\infty$, and let $\Phi$ and $\varphi$ satisfy
	\eqref{Ass1} and \eqref{Ass2}, respectively. For
	$k\in\mathbb{N}$, define
	$\varphi_{k}=2^{kn}\varphi(2^{k}\cdot).$
	\medskip
	The \emph{mixed-norm Herz-type Triebel--Lizorkin space}
	$\dot{E}_{\vec{p}}^{\vec{\alpha},\vec{q}}F_{\beta}^{s}$
	is defined as the collection of all
	$f\in\mathcal{S}'(\mathbb{R}^{n})$ such that
	\[
	\|f\|_{\dot{E}_{\vec{p}}^{\vec{\alpha},\vec{q}}F_{\beta}^{s}}
	=
	\left\|
	\left(
	\sum_{k=0}^{\infty}
	2^{ks\beta}
	|\varphi_{k}*f|^{\beta}
	\right)^{1/\beta}
	\right\|_{\dot{E}_{\vec{p}}^{\vec{\alpha},\vec{q}}
		(\mathbb{R}^{n})}
	<\infty,
	\]
	where $\varphi_{0}$ is replaced by $\Phi$, with the usual modification
	when $\beta=\infty$.
\end{definition}

Let $s\in\mathbb{R}$, $0<\vec{p}<\infty$, and
$0<\beta\le\infty$. Using the system
$\{\varphi_{k}\}_{k\in\mathbb{N}_{0}}$, we define the quasi-norms
\[
\|f\|_{F_{\vec{p},\beta}^{s}}
=
\left\|
\left(
\sum_{k=0}^{\infty}
2^{ks\beta}
|\varphi_{k}*f|^{\beta}
\right)^{1/\beta}
\right\|_{L^{\vec{p}}(\mathbb{R}^{n})}
 .
\]

The mixed-norm Triebel--Lizorkin  space
$F_{\vec{p},\beta}^{s}$ consists of all
$f\in\mathcal{S}'(\mathbb{R}^{n})$ for which
$\|f\|_{F_{\vec{p},\beta}^{s}}<\infty$.

It is well known that these spaces are independent of the particular
choice of the system
$\{\varphi_{k}\}_{k\in\mathbb{N}_{0}}$, up to equivalence of
quasi-norms. Further details concerning the classical theory of these
spaces, including the homogeneous case, may be found in
\cite{GJN17,JS07,JS08,JHS12}. For the classical 
Triebel--Lizorkin spaces, we refer the reader to
\cite{FJ86,FJ90,FrJaWe01,T83,T2}.

It is immediate from the definitions that, when
$\vec{\alpha}=0$ and $\vec{p}=\vec{q}$,
\[
\dot{E}_{\vec{p}}^{0,\vec{p}}F_{\beta}^{s}
=
F_{\vec{p},\beta}^{s}.
\]

Let $\Phi$ and $\varphi$ satisfy \eqref{Ass1} and \eqref{Ass2},
respectively. By \cite[Section 12]{FJ90}, there exist functions
$\Psi\in\mathcal{S}(\mathbb{R}^{n})$ satisfying \eqref{Ass1} and
$\psi\in\mathcal{S}(\mathbb{R}^{n})$ satisfying \eqref{Ass2} such that
\begin{equation}
\mathcal{F}\widetilde{\Phi}(\xi)\mathcal{F}\Psi(\xi)
+
\sum_{k=1}^{\infty}
\mathcal{F}\widetilde{\varphi}(2^{-k}\xi)
\mathcal{F}\psi(2^{-k}\xi)
=
1,
\qquad
\xi\in\mathbb{R}^{n}.
\label{Ass3}
\end{equation}

A fundamental tool in the study of the above spaces is the following
Calder\'{o}n reproducing formula; see \cite[(12.4)]{FJ90} and
\cite[Lemma 2.3]{YSY10}.

\begin{lemma}
	\label{DW-lemma1}
	Let $\Phi,\Psi\in\mathcal{S}(\mathbb{R}^{n})$ satisfy \eqref{Ass1},
	and let $\varphi,\psi\in\mathcal{S}(\mathbb{R}^{n})$ satisfy
	\eqref{Ass2}. Assume that \eqref{Ass3} holds. Then, for every
	$f\in\mathcal{S}'(\mathbb{R}^{n})$,
	\begin{align}
	f
	&=
	\widetilde{\Phi}*\Psi*f
	+\sum_{k=1}^{\infty}
	\widetilde{\varphi}_{k}*\psi_{k}*f
	\notag\\
	&=
	\sum_{m\in\mathbb{Z}^{n}}
	\widetilde{\Phi}*f(m)\,\Psi_{m}
	+
	\sum_{k=1}^{\infty}
	2^{-kn/2}
	\sum_{m\in\mathbb{Z}^{n}}
	\widetilde{\varphi}_{k}*f(2^{-k}m)\,
	\psi_{k,m},
	\label{proc2}
	\end{align}
	with convergence in $\mathcal{S}'(\mathbb{R}^{n})$, where
	\[
	\Psi_{m}=\Psi(\cdot-m),
	\qquad
	\psi_{k,m}
	=
	2^{kn/2}\psi(2^{k}\cdot-m),
	\qquad
	m\in\mathbb{Z}^{n},\ k\in\mathbb{N}.
	\]
\end{lemma}
Let $\Phi,\Psi,\varphi,\psi\in\mathcal{S}(\mathbb{R}^{n})$ satisfy
\eqref{Ass1}--\eqref{Ass3}. The $\varphi$-transform
$S_{\varphi}$ is defined by
\[
(S_{\varphi}f)_{0,m}
=
\langle f,\Phi_{m}\rangle,
\qquad
(S_{\varphi}f)_{k,m}
=
\langle f,\varphi_{k,m}\rangle,
\]
where
\[
\Phi_{m}=\Phi(\cdot-m),
\qquad
\varphi_{k,m}
=
2^{kn/2}\varphi(2^{k}\cdot-m),
\quad
m\in\mathbb{Z}^{n},
\ k\in\mathbb{N}.
\]

The inverse $\varphi$-transform $T_{\psi}$ is defined by
\[
T_{\psi}\lambda
=
\sum_{m\in\mathbb{Z}^{n}}
\lambda_{0,m}\Psi_{m}
+
\sum_{k=1}^{\infty}
\sum_{m\in\mathbb{Z}^{n}}
\lambda_{k,m}\psi_{k,m},
\]
where
$\lambda=\{\lambda_{k,m}\}_{k\in\mathbb{N}_{0},\,m\in\mathbb{Z}^{n}}
\subset\mathbb{C}$; see \cite[p.~131]{FJ90}.

We now introduce the sequence spaces associated with $\dot{E}_{\vec{p}}^{\vec{\alpha},\vec{q}}F_{\beta}^{s}$.

\begin{definition}
	\label{sequence-space}
	Let
	$s\in\mathbb{R}$,
	$0<\vec{p},\vec{q}<\infty$,
	$\vec{\alpha}=(\alpha_{1},\ldots,\alpha_{n})\in\mathbb{R}^{n}$,
	and $0<\beta\le\infty$. The space
	$\dot{E}_{\vec{p}}^{\vec{\alpha},\vec{q}}f_{\beta}^{s}$
	is defined as the collection of all complex-valued sequences
	$\lambda=\{\lambda_{k,m}\}_{k\in\mathbb{N}_{0},\,m\in\mathbb{Z}^{n}}$
	such that
	\[
	\|\lambda\|_{\dot{E}_{\vec{p}}^{\vec{\alpha},\vec{q}}f_{\beta}^{s}}
	=
	\left\|
	\left(
	\sum_{v=0}^{\infty}
	\sum_{m\in\mathbb{Z}^{n}}
	2^{v(s+\frac{n}{2})\beta}
	|\lambda_{v,m}|^{\beta}
	\chi_{v,m}
	\right)^{1/\beta}
	\right\|_{\dot{E}_{\vec{p}}^{\vec{\alpha},\vec{q}}(\mathbb{R}^{n})}
	<\infty.
	\]
\end{definition}

Observe that the spaces
$\dot{E}_{\vec{p}}^{\vec{\alpha},\vec{q}}F_{\beta}^{s}$
are quasi-normed. If
\[
d=\min\{1,p_{-},q_{-},\beta\},
\]
then
\[
\|f+g\|_{\dot{E}_{\vec{p}}^{\vec{\alpha},\vec{q}}F_{\beta}^{s}}
\lesssim
\|f\|_{\dot{E}_{\vec{p}}^{\vec{\alpha},\vec{q}}F_{\beta}^{s}}
+
\|g\|_{\dot{E}_{\vec{p}}^{\vec{\alpha},\vec{q}}F_{\beta}^{s}}
\]
for all
$f,g\in
\dot{E}_{\vec{p}}^{\vec{\alpha},\vec{q}}F_{\beta}^{s}$.

The following lemma shows that the inverse $\varphi$-transform is
well defined on
$\dot{E}_{\vec{p}}^{\vec{\alpha},\vec{q}}f_{\beta}^{s}$.

\begin{lemma}
	\label{Inv-phi-trans}
	Let
	$s\in\mathbb{R}$,
	$0<\vec{p},\vec{q}<\infty$,
	$\vec{\alpha}=(\alpha_{1},\ldots,\alpha_{n})\in\mathbb{R}^{n}$,
	$0<\beta\le\infty$,
	and assume that
	\[
	\alpha_{i}>-\frac1{p_{i}},
	\qquad
	i\in \{1,\dots,n\}.
	\]
	Let $\Psi$ and $\psi$ satisfy
	\eqref{Ass1} and \eqref{Ass2}, respectively.
	Then, for every
	$\lambda\in
	\dot{E}_{\vec{p}}^{\vec{\alpha},\vec{q}}f_{\beta}^{s}$,
	the series
	\[
	T_{\psi}\lambda
	=
	\sum_{m\in\mathbb{Z}^{n}}
	\lambda_{0,m}\Psi_{m}
	+
	\sum_{k=1}^{\infty}
	\sum_{m\in\mathbb{Z}^{n}}
	\lambda_{k,m}\psi_{k,m}
	\]
	converges in
	$\mathcal{S}'(\mathbb{R}^{n})$.
	Moreover,
	\[
	T_{\psi}:
	\dot{E}_{\vec{p}}^{\vec{\alpha},\vec{q}}f_{\beta}^{s}
	\longrightarrow
	\mathcal{S}'(\mathbb{R}^{n})
	\]
	is continuous.
\end{lemma}
\begin{proof}
	
	Choose
	\[
	0<r_i<
	\min\left\{
	\frac{1}{\alpha_i+\frac1{p_i}},
	\,p_i,
	\,q_i
	\right\},
	\qquad i\in \{1,\dots,n\},
	\]
	and let
	\[
	\frac1{\vec r}
	=
	\frac1{\vec p}
	+
	\frac1{\vec t},
	\qquad
	\frac1{\vec r}
	=
	\frac1{\vec q}
	+
	\frac1{\vec v},
	\]
	where
	$\vec r,\vec t,\vec v\in(0,\infty)^n$.
	Let
	$\lambda\in
	\dot{E}_{\vec p}^{\vec\alpha,\vec q}f_\beta^s$
	and
	$\varphi\in\mathcal S(\mathbb R^n)$.
	Set
	\[
	I_1
	=
	\sum_{m\in\mathbb Z^n}
	|\lambda_{0,m}|
	\,|\langle \Psi_m,\varphi\rangle|,
	\qquad
	I_2
	=
	\sum_{k=1}^{\infty}
	\sum_{m\in\mathbb Z^n}
	|\lambda_{k,m}|
	\,|\langle \psi_{k,m},\varphi\rangle|.
	\]
	
	It suffices to show that there exist
	$M\in\mathbb N$ and a constant $c>0$ such that
	\[
	I_1+I_2
	\le
	c\,\|\varphi\|_{\mathcal S_M}
	\|\lambda\|_{\dot E_{\vec p}^{\vec\alpha,\vec q}f_\beta^s}.
	\]
	
	\medskip
	
	\noindent
	\textit{Estimate of $I_1$.}
	Let
	$m\in\mathbb Z^n$ and choose
	$L,M\in\mathbb N$ with
	$M>L+n$.
	Since
	$\varphi,\Psi\in\mathcal S(\mathbb R^n)$,
	we obtain
	\begin{align*}
	|\langle \Psi_m,\varphi\rangle|
	&\le
	\int_{\mathbb R^n}
	|\Psi(x-m)|\,|\varphi(x)|\,dx
	\\
	&\le
	\|\varphi\|_{\mathcal S_M}
	\|\Psi\|_{\mathcal S_L}
	\int_{\mathbb R^n}
	(1+|x-m|)^{-L-n}
	(1+|x|)^{-M-n}\,dx
	\\
	&\lesssim
	\|\varphi\|_{\mathcal S_M}
	\|\Psi\|_{\mathcal S_L}
	(1+|m|)^{-L-n}.
	\end{align*}
	
	The last estimate follows from
	\[
	(1+|x-m|)^{-L-n}
	\le
	(1+|m|)^{-L-n}(1+|x|)^{L+n},
	\qquad
	x\in\mathbb R^n,\ m\in\mathbb Z^n.
	\]
	
We claim that
\begin{equation}
|\gamma|
\lesssim
\prod_{i=1}^{n}(1+|m_i|)^{\frac1{t_i}-\alpha_i}
\,
\Big\|
\cdots
\big\|
\big\|
\gamma \chi_{0,m_1}
\big\|_{\dot K_{p_1}^{\alpha_1,q_1}}
\chi_{0,m_2}
\big\|_{\dot K_{p_2}^{\alpha_2,q_2}}
\cdots
\chi_{0,m_n}
\Big\|_{\dot K_{p_n}^{\alpha_n,q_n}},
\label{main-est}
\end{equation}
for every $\gamma\in\mathbb C$ such that the right-hand side of
\eqref{main-est} is finite. Taking $\gamma=\lambda_{0,m}$ in
\eqref{main-est}, we obtain
\begin{equation*}
|\lambda_{0,m}|
\lesssim
\prod_{i=1}^{n}(1+|m_i|)^{\frac1{t_i}-\alpha_i}
\,
\|\lambda\|_{\dot E_{\vec p}^{\vec\alpha,\vec q}f_\beta^s}.
\end{equation*}

We prove \eqref{main-est} by induction on $n$.
First, let $n=1$. By H\"{o}lder's inequality,
\begin{align}
|\gamma|^{r_1}
&=
\frac1{|Q_{0,m_1}|}
\sum_{j_1=-\infty}^{\infty}
\|\gamma\chi_{0,m_1}\chi_{R_{j_1}}\|_{L^{{r_{1}}}(\mathbb R)}^{\,r_1}
\notag\\
&\le
c
\sum_{j_1=-\infty}^{\infty}
\|\gamma\chi_{0,m_1}\chi_{R_{j_1}}\|_{L^{{p_{1}}}(\mathbb R)}^{\,r_1}
\,
\|\chi_{0,m_1}\chi_{R_{j_1}}\|_{L^{{t_{1}}}(\mathbb R)}^{\,r_1}
\notag\\
&\le
c
\sum_{j_1=-\infty}^{\infty}
\|\gamma\chi_{0,m_1}\chi_{R_{j_1}}\|_{L^{{p_{1}}}(\mathbb R)}^{\,r_1}
\,
\|\chi_{R_{j_1}}\|_{L^{{t_{1}}}(\mathbb R)}^{\,r_1},
\label{sum1}
\end{align}
where the positive constant $c$ is independent of $m_1\in\mathbb Z$ and $\gamma$.

Since
\[
\chi_{0,m_1}\chi_{R_{j_1}}=0
\quad\text{whenever}\quad
2^{j_1-1}>1+|m_1|,
\]
the sum in \eqref{sum1} reduces to
\begin{equation}
\sum_{\substack{j_1\in\mathbb Z\\ 2^{j_1-1}\le 1+|m_1|}}
\|\gamma\chi_{0,m_1}\chi_{R_{j_1}}\|_{L^{{p_{1}}}(\mathbb R)}^{\,r_1}
\,
\|\chi_{R_{j_1}}\|_{L^{{t_{1}}}(\mathbb R)}^{\,r_1}.
\label{sum2}
\end{equation}

Applying H\"{o}lder's inequality in sequence spaces, we obtain
\begin{align}
\eqref{sum2}
&\le
\Bigg(
\sum_{\substack{j_1\in\mathbb Z\\ 2^{j_1-1}\le 1+|m_1|}}
2^{j_1\alpha_1 q_1}
\|\gamma\chi_{0,m_1}\chi_{R_{j_1}}\|_{L^{{p_{1}}}(\mathbb R)}^{\,q_1}
\Bigg)^{r_1/q_1}
\notag\\
&\qquad\times
\Bigg(
\sum_{\substack{j_1\in\mathbb Z\\ 2^{j_1-1}\le 1+|m_1|}}
2^{-j_1\alpha_1 v_1}
\|\chi_{R_{j_1}}\|_{L^{{t_{1}}}(\mathbb R)}^{\,v_1}
\Bigg)^{r_1/v_1}
\notag\\
&\le
c
\|\gamma\chi_{0,m_1}\|_{\dot K_{p_1}^{\alpha_1,q_1}}^{\,r_1}
\Bigg(
\sum_{\substack{j_1\in\mathbb Z\\ 2^{j_1-1}\le 1+|m_1|}}
2^{j_1(\frac1{t_1}-\alpha_1)v_1}
\Bigg)^{r_1/v_1}
\notag\\
&\le
c
(1+|m_1|)^{(\frac1{t_1}-\alpha_1)r_1}
\,
\|\gamma\chi_{0,m_1}\|_{\dot K_{p_1}^{\alpha_1,q_1}}^{\,r_1},
\label{sum3}
\end{align}
since $\frac1{t_1}-\alpha_1>0$.

Combining \eqref{sum1} and \eqref{sum3}, we conclude that
\[
|\gamma|
\lesssim
(1+|m_1|)^{\frac1{t_1}-\alpha_1}
\,
\|\gamma\chi_{0,m_1}\|_{\dot K_{p_1}^{\alpha_1,q_1}}.
\]

This proves \eqref{main-est} for $n=1$. The general case follows by induction on $n$.

Finally, choosing
\[
L>\boldsymbol{\alpha}-\frac1{\mathbf t}+n,
\]
we obtain
\[
I_1
\lesssim
\|\varphi\|_{\mathcal S_M}
\,
\|\lambda\|_{\dot E_{\vec p}^{\vec\alpha,\vec q}f_\beta^s}.
\]
	
	\medskip
	
	\noindent
	\textit{Estimate of $I_2$.}
	We recall the following estimate (see \cite[Lemma 2.4]{YSY10}):
	since $\psi$ has vanishing moments of arbitrary order, for every
	$L,M>0$ there exists a constant
	$C=C(M,n)>0$ such that
	\begin{equation}
	|\psi_k*\varphi(x)|
	\le
	C\,2^{-kL}
	\|\psi\|_{\mathcal S_{M+1}}
	\|\varphi\|_{\mathcal S_{M+1}}
	(1+|x|)^{-n-L},
	\label{convolution}
	\end{equation}
	for all $k\in\mathbb N$ and
	$x\in\mathbb R^n$.
	
	Let
	$\breve\varphi=\varphi(-\cdot)$.
	By \eqref{convolution},
	\begin{align*}
	|\langle\psi_{k,m},\varphi\rangle|
	&=
	2^{-k\frac{n}{2}}
	|\psi_k*\breve\varphi(-2^{-k}m)|
	\\
	&\lesssim
	2^{-k(\frac n2+L)}
	\|\psi\|_{\mathcal S_{M+1}}
	\|\varphi\|_{\mathcal S_{M+1}}
	(1+|2^{-k}m|)^{-n-L}.
	\end{align*}
	
	Proceeding exactly as in the estimate of $I_1$, we obtain
	\[
	|\lambda_{k,m}|
	\lesssim
	2^{\frac{k}{\mathbf r}}
	(1+|2^{-k}m|)^{\frac1{\mathbf t}-\boldsymbol{\alpha}}
	\|\lambda\|_
	{\dot E_{\vec p}^{\vec\alpha,\vec q}f_\beta^s}.
	\]
	
	Consequently,
	\begin{align*}
	I_2
	&\lesssim
	\|\varphi\|_{\mathcal S_{M+1}}
	\|\psi\|_{\mathcal S_{M+1}}
	\|\lambda\|_
	{\dot E_{\vec p}^{\vec\alpha,\vec q}f_\beta^s}
	\\
	&\qquad\times
	\sum_{k=1}^{\infty}
	\sum_{m\in\mathbb Z^n}
	2^{k(\frac1{\mathbf r}-s-n-L)}
	(1+|2^{-k}m|)^{
		\frac1{\mathbf t}
		-\boldsymbol{\alpha}
		-n-L}.
	\end{align*}
	
	Choosing $L$ sufficiently large, the above series converges, and hence
	\[
	I_2
	\lesssim
	\|\varphi\|_{\mathcal S_{M+1}}
	\|\psi\|_{\mathcal S_{M+1}}
	\|\lambda\|_
	{\dot E_{\vec p}^{\vec\alpha,\vec q}f_\beta^s}.
	\]
	
	Combining the estimates for $I_1$ and $I_2$ completes the proof.
\end{proof}

For a sequence $\lambda =\{\lambda _{k,m}\}_{k\in \mathbb{N}_{0},m\in 
\mathbb{Z}^{n}}\subset \mathbb{C},0<r\leq \infty $ and a fixed $d>0$, set%
\begin{equation*}
\lambda _{k,m,r,d}^{\ast }=\Big(\sum_{h\in \mathbb{Z}^{n}}\frac{|\lambda
_{k,h}|^{r}}{(1+2^{k}|2^{-k}h-2^{-k}m|)^{d}}\Big)^{1/r}
\end{equation*}%
and $\lambda _{r,d}^{\ast }:=\{\lambda _{k,m,r,d}^{\ast }\}_{k\in \mathbb{N}%
_{0},m\in \mathbb{Z}^{n}}\subset \mathbb{C}$ with the usual modification if $%
r=\infty $.

Similarly to the classical  Triebel--Lizorkin spaces (see \cite{FJ90}), we obtain the following estimate.

\begin{lemma}
	\label{lamda-equi copy1}
	Let $s\in \mathbb{R}$, $0<\vec p,\vec q< \infty$,
	$\vec\alpha=(\alpha_1,\ldots,\alpha_n)\in\mathbb{R}^n$,
	$0<\beta\le\infty$, and $i\in\{1,\ldots,n\}$. Assume that
	\[
	0<a<\min\Bigl(1,p_{-},q_{-},\beta,
	\frac{1}{\alpha_i+\frac1{p_i}}\Bigr),
	\]
	and define
	\[
	r=\min(1,p_{-},q_{-},\beta).
	\]
	If $d>\frac{nr}{a}$, then
	\[
	\lambda_{k,m,r,d}^{*}\,\chi_{k,m}
	\lesssim
	\mathcal{M}_{a}
	\Bigl(
	\sum_{h\in\mathbb Z^{n}}
	\lambda_{k,h}\chi_{k,h}
	\Bigr),
	\qquad
	k\in\mathbb N_{0},\; m\in\mathbb Z^{n}.
	\]
\end{lemma}
Combining Lemma \ref{lamda-equi copy1} with Lemma \ref{Maximal-Inq}, and arguing as in the classical theory of Triebel-Lizorkin spaces (see \cite{FJ90}), we obtain the following result.

\begin{lemma}
	\label{lamda-equi}
	Let $s\in \mathbb{R}$, $0<\vec{p},\vec{q}<\infty$,
	$\vec{\alpha}=(\alpha_{1},\ldots,\alpha_{n})\in\mathbb{R}^{n}$,
	$0<\beta\leq\infty$, and
	$\alpha_i>-\frac1{p_i}$ for all $i\in\{1,\ldots,n\}$.
	Assume that
	\[
	0<a<
	\min\Bigl(1,p_{-},q_{-},\beta,
	\frac{1}{\alpha_i+\frac1{p_i}}\Bigr),
	\]
	and let
	\[
	r= \min(1,p_{-},q_{-},\beta),
\]
	If $d>\frac{nr}{a}$, then
	\[
	\bigl\|\lambda_{r,d}^{*}\bigr\|_
	{\dot{E}_{\vec p}^{\vec\alpha,\vec q}f_{\beta}^{s}}
	\approx
	\bigl\|\lambda\bigr\|_
	{\dot{E}_{\vec p}^{\vec\alpha,\vec q}f_{\beta}^{s}}.
	\]
\end{lemma}

We are now in a position to state the following theorem, known as the
$\varphi$-transform characterization in the sense of Frazier and Jawerth.
Its proof follows by a straightforward adaptation of
\cite[Theorem 2.2]{FJ90}, together with Lemma \ref{lamda-equi}.

\begin{theorem}
	\label{phi-tran}
	Let $s\in \mathbb{R}$, $0<\vec{p},\vec{q}<\infty$,
	$\vec{\alpha}=(\alpha_{1},\ldots,\alpha_{n})\in\mathbb{R}^{n}$,
	$0<\beta\leq\infty$, and
	$\alpha_i>-\frac1{p_i}$ for all $i\in\{1,\ldots,n\}$.
	Let $\Phi,\Psi\in\mathcal{S}(\mathbb{R}^{n})$ satisfy \eqref{Ass1},
	and let $\varphi,\psi\in\mathcal{S}(\mathbb{R}^{n})$ satisfy
	\eqref{Ass2} such that \eqref{Ass3} holds.
	
	Then the operators
	\[
	S_{\varphi}:
	\dot{E}_{\vec p}^{\vec\alpha,\vec q}F_{\beta}^{s}
	\longrightarrow
	\dot{E}_{\vec p}^{\vec\alpha,\vec q}f_{\beta}^{s}
	\]
	and
	\[
	T_{\psi}:
	\dot{E}_{\vec p}^{\vec\alpha,\vec q}f_{\beta}^{s}
	\longrightarrow
	\dot{E}_{\vec p}^{\vec\alpha,\vec q}F_{\beta}^{s}
	\]
	are bounded. Moreover,
	\[
	T_{\psi}\circ S_{\varphi}
	=
	I
	\]
	on $\dot{E}_{\vec p}^{\vec\alpha,\vec q}F_{\beta}^{s}$.
\end{theorem}

\begin{remark}
	Theorem \ref{phi-tran} provides a powerful tool for studying the spaces
	$\dot{E}_{\vec p}^{\vec\alpha,\vec q}F_{\beta}^{s}$.
	Indeed, many problems can be transferred from the function space to the
	corresponding sequence space, where they are often easier to handle.
	
	More precisely, under the assumptions of Theorem \ref{phi-tran}, one has
	\[
	\bigl\|
	\{\langle f,\varphi_{k,m}\rangle\}_{k\in\mathbb N_{0},\,m\in\mathbb Z^{n}}
	\bigr\|_
	{\dot{E}_{\vec p}^{\vec\alpha,\vec q}f_{\beta}^{s}}
	\approx
	\|f\|_
	{\dot{E}_{\vec p}^{\vec\alpha,\vec q}F_{\beta}^{s}}.
	\]
\end{remark}
Following the arguments in \cite{Dr-AOT23}, we obtain the following result.

\begin{lemma}
	Let $s\in \mathbb{R}$, $0<\vec{p},\vec{q}< \infty$, $\vec{\alpha}=(\alpha
	_{1},\ldots,\alpha _{n})\in \mathbb{R}^{n}$, $0<\beta \leq \infty$, and
	$\alpha _{i}>-\frac{1}{p_{i}}$ for all $i\in \{1,\ldots,n\}$. Then the spaces
	$\dot{E}_{\vec{p}}^{\vec{\alpha},\vec{q}}f_{\beta }^{s}$ are quasi-Banach
	spaces. Moreover, they are Banach spaces whenever 	$1\leq \vec{p},\vec{q}< \infty$ and 
	$1\leq \beta \leq \infty$.
\end{lemma}

Combining this lemma with Theorem \ref{phi-tran} and the results of
\cite{GJN17}, we obtain the following useful property of these function
spaces.

\begin{theorem}
	Let $s\in \mathbb{R}$, $0<\vec{p},\vec{q}< \infty$, $\vec{\alpha}=(\alpha
	_{1},\ldots,\alpha _{n})\in \mathbb{R}^{n}$, $0<\beta \leq \infty$, and
	$\alpha _{i}>-\frac{1}{p_{i}}$ for all $i\in \{1,\ldots,n\}$. Then the spaces
	$\dot{E}_{\vec{p}}^{\vec{\alpha},\vec{q}}F_{\beta }^{s}$ are quasi-Banach
	spaces. Moreover, they are Banach spaces whenever
		$1\leq \vec{p},\vec{q}< \infty$ and 
	$1\leq \beta \leq \infty$.
\end{theorem}

Let $\vartheta$ be a function in $\mathcal{S}(\mathbb{R}^{n})$ satisfying
\begin{equation}
\vartheta(x)=1 \quad \text{for } |x|\leq 1,
\qquad
\vartheta(x)=0 \quad \text{for } |x|\geq \frac{3}{2}.
\label{function-v}
\end{equation}
Define
\[
\psi_{0}(x)=\vartheta(x), \qquad
\psi_{1}(x)=\vartheta\!\left(\frac{x}{2}\right)-\vartheta(x),
\]
and
\[
\psi_{k}(x)=\psi_{1}(2^{-k+1}x), \qquad k=2,3,\ldots .
\]
Then
\[
\operatorname{supp}\psi_{k}
\subset
\bigl\{x\in\mathbb{R}^{n}:2^{k-1}\leq |x|
\leq 3\cdot 2^{k-1}\bigr\},
\]
and
\begin{equation}
\sum_{k=0}^{\infty}\psi_{k}(x)=1,
\qquad x\in\mathbb{R}^{n}.
\label{partition}
\end{equation}
The family $\{\psi_{k}\}_{k\in\mathbb{N}_{0}}$ is called a smooth dyadic
resolution of unity. Consequently, every
$f\in\mathcal{S}'(\mathbb{R}^{n})$ admits the Littlewood--Paley
decomposition
\[
f=\sum_{k=0}^{\infty}\mathcal{F}^{-1}\psi_{k}*f,
\]
where the series converges in $\mathcal{S}'(\mathbb{R}^{n})$.

We define

\[
\|f\|_{\dot{E}_{\vec{p}}^{\vec{\alpha},\vec{q}}F_{\beta}^{s}}^{\psi_{0},\psi_{1}}
=
\Bigg\|
\Bigg(
\sum_{k=0}^{\infty}
2^{ks\beta}
\bigl|
\mathcal{F}^{-1}\psi_{k}*f
\bigr|^{\beta}
\Bigg)^{1/\beta}
\Bigg\|_{\dot{E}_{\vec{p}}^{\vec{\alpha},\vec{q}}(\mathbb{R}^{n})}.
\]

The following theorem is proved in the same spirit as the corresponding
result for classical  Triebel--Lizorkin spaces, with the aid of
Lemma \ref{Maximal-Inq}.

\begin{theorem}
	\label{partition-equi}
	Let $s\in \mathbb{R}$, $0<\vec{p},\vec{q}< \infty$,
	$\vec{\alpha}=(\alpha _{1},\ldots,\alpha _{n})\in \mathbb{R}^{n}$,
	$0<\beta \leq \infty$, and
	$\alpha _{i}>-\frac{1}{p_{i}}$ for all $i\in \{1,\ldots,n\}$.
	Then a tempered distribution
	$f$ belongs to
	$\dot{E}_{\vec{p}}^{\vec{\alpha},\vec{q}}F_{\beta}^{s}$
	if and only if
	\[
	\|f\|_{\dot{E}_{\vec{p}}^{\vec{\alpha},\vec{q}}F_{\beta}^{s}}^{\psi_{0},\psi_{1}}
	<\infty.
	\]
	Furthermore, the quasi-norms
	\[
	\|f\|_{\dot{E}_{\vec{p}}^{\vec{\alpha},\vec{q}}F_{\beta}^{s}}
	\quad\text{and}\quad
	\|f\|_{\dot{E}_{\vec{p}}^{\vec{\alpha},\vec{q}}F_{\beta}^{s}}^{\psi_{0},\psi_{1}}
	\]
	are equivalent.
\end{theorem}
\section{Elementary Embeddings}
The following theorem provides basic embeddings for the spaces
$\dot{E}_{\vec{p}}^{\vec{\alpha},\vec{q}}F_{\beta}^{s}$.

\begin{theorem}
	\label{embeddings1.1}
	Let $s\in\mathbb{R}$, $0<\vec{p},\vec{q}<\infty$,
	$\vec{\alpha}=(\alpha_1,\ldots,\alpha_n)\in\mathbb{R}^{n}$,
	$0<\beta\leq\infty$, and
	$\alpha_i>-\frac1{p_i}$ for all $i\in\{1,\ldots,n\}$.
	
	\begin{enumerate}
		\item[(i)] If $0<\beta_1\leq\beta_2\leq\infty$, then
		\begin{equation}
		\dot{E}_{\vec{p}}^{\vec{\alpha},\vec{q}}F_{\beta_1}^{s}
		\hookrightarrow
		\dot{E}_{\vec{p}}^{\vec{\alpha},\vec{q}}F_{\beta_2}^{s}.
		\label{embed1}
		\end{equation}
		
		\item[(ii)] If $0<\beta_1,\beta_2\leq\infty$ and $\varepsilon>0$, then
		\begin{equation}
		\dot{E}_{\vec{p}}^{\vec{\alpha},\vec{q}}F_{\beta_1}^{s+\varepsilon}
		\hookrightarrow
		\dot{E}_{\vec{p}}^{\vec{\alpha},\vec{q}}F_{\beta_2}^{s}.
		\label{embed2}
		\end{equation}
		
		\item[(iii)] Let
		$\vec{q}_1=(q_1^1,\ldots,q_n^1)\in(0,\infty)^n$ and
		$\vec{q}_2=(q_1^2,\ldots,q_n^2)\in(0,\infty)^n$.
		If $0<q_i^1\leq q_i^2<\infty$ for all $i\in\{1,\ldots,n\}$, then
		\begin{equation}
		\dot{E}_{\vec{p}}^{\vec{\alpha},\vec{q}_1}F_{\beta}^{s}
		\hookrightarrow
		\dot{E}_{\vec{p}}^{\vec{\alpha},\vec{q}_2}F_{\beta}^{s}.
		\label{embed3}
		\end{equation}
		
		\item[(iv)] Let
		$\vec{p}_1=(p_1^1,\ldots,p_n^1)\in(0,\infty)^n$ and
		$\vec{p}_2=(p_1^2,\ldots,p_n^2)\in(0,\infty)^n$.
		If $0<p_i^1\leq p_i^2<\infty$ for all $i\in\{1,\ldots,n\}$, then
		\begin{equation}
		\dot{E}_{\vec{p}_2}^{\vec{\alpha},\vec{q}}F_{\beta}^{s}
		\hookrightarrow
		\dot{E}_{\vec{p}_1}^{\vec{r},\vec{q}}F_{\beta}^{s},
		\label{embed4}
		\end{equation}
		where
		\[
		r_i=\alpha_i-\Big(\frac1{p_i^1}-\frac1{p_i^2}\Big),
		\qquad i\in\{1,\ldots,n\}.
		\]
	\end{enumerate}
\end{theorem}

\begin{proof}
	The embeddings \eqref{embed1} and \eqref{embed3} are immediate consequences
	of the embeddings between Lebesgue sequence spaces and the embedding
	\[
	\dot{E}_{\vec{p}}^{\vec{\alpha},\vec{q}_1}(\mathbb{R}^n)
	\hookrightarrow
	\dot{E}_{\vec{p}}^{\vec{\alpha},\vec{q}_2}(\mathbb{R}^n).
	\]
	
	Let $\Phi$ and $\varphi$ satisfy \eqref{Ass1} and \eqref{Ass2},
	respectively, and let
	$f\in \dot{E}_{\vec{p}}^{\vec{\alpha},\vec{q}}F_{\beta_1}^{s+\varepsilon}$.
	To prove \eqref{embed2}, note that $\varepsilon>0$ implies
	\[
	\Big\|
	\Big(
	\sum_{k=0}^{\infty}
	2^{ks\beta_2}
	|\varphi_k*f|^{\beta_2}
	\Big)^{1/\beta_2}
	\Big\|_{\dot{E}_{\vec{p}}^{\vec{\alpha},\vec{q}}(\mathbb{R}^n)}
	\leq
	c
	\Big\|
	\sup_{k\in\mathbb{N}_0}
	\big(
	2^{k(s+\varepsilon)}
	|\varphi_k*f|
	\big)
	\Big\|_{\dot{E}_{\vec{p}}^{\vec{\alpha},\vec{q}}(\mathbb{R}^n)}.
	\]
	The desired estimate then follows from the embedding
	$\ell^{\beta_1}\hookrightarrow\ell^{\infty}$.

	Finally, the embedding \eqref{embed4} follows immediately from
	H\"{o}lder's inequality.
\end{proof}

\begin{theorem}
	\label{embeddings-S-inf}
	Let $s\in \mathbb{R}$, $0<\vec{p},\vec{q}<\infty$,
	$\vec{\alpha}=(\alpha_{1},\ldots,\alpha_{n})\in \mathbb{R}^{n}$,
	$0<\beta\leq\infty$, and
	$\alpha_i>-\frac1{p_i}$ for all $i\in\{1,\ldots,n\}$.
	
	\smallskip
	
	\noindent
	$\mathrm{(i)}$ The embedding
	\begin{equation}
	\mathcal{S}(\mathbb{R}^{n})
	\hookrightarrow
	\dot{E}_{\vec{p}}^{\vec{\alpha},\vec{q}}F_{\beta}^{s}
	\label{embedding}
	\end{equation}
	holds. Moreover, if 
	$0<\beta<\infty$, then $\mathcal{S}(\mathbb{R}^{n})$ is dense in
	$\dot{E}_{\vec{p}}^{\vec{\alpha},\vec{q}}F_{\beta}^{s}$.
	
	\smallskip
	
	\noindent
	$\mathrm{(ii)}$ The embedding
	\begin{equation}
	\dot{E}_{\vec{p}}^{\vec{\alpha},\vec{q}}F_{\beta}^{s}
	\hookrightarrow
	\mathcal{S}'(\mathbb{R}^{n})
	\label{embeddingsSch}
	\end{equation}
	holds.
\end{theorem}

\begin{proof}
	The proof of the embedding \eqref{embedding} and the density of
	$\mathcal{S}(\mathbb{R}^{n})$ in
	$\dot{E}_{\vec{p}}^{\vec{\alpha},\vec{q}}F_{\beta}^{s}$
	follows from the same arguments as in
	\cite{Dr-AOT23,YY1}. Therefore, it remains to prove that
	\[
	\dot{E}_{\vec{p}}^{\vec{\alpha},\vec{q}}F_{\beta}^{s}
	\hookrightarrow
	\mathcal{S}'(\mathbb{R}^{n}).
	\]
	But this follows by the embedding
	\[
	\dot{E}_{\vec{p}}^{\vec{\alpha},\vec{q}}B_{\infty}^{s}
	\hookrightarrow
	\mathcal{S}'(\mathbb{R}^{n}),
	\]
	see \cite{Dr-MixedBesov}.
	This completes the proof.
\end{proof}

\section{Sobolev embeddings}

In this section, we establish Sobolev-type embeddings for the spaces
$\dot{E}_{\vec p}^{\vec\alpha,\vec q}F_\beta^s$. It is well known that
\begin{equation}
F_{q,\infty}^{s_{2}}\hookrightarrow F_{p,\beta}^{s_{1}},
\label{Sobolev-Tr-Li}
\end{equation}
provided that
\[
s_{1}-\frac{n}{p}=s_{2}-\frac{n}{q},
\]
where $0<q<p<\infty$ and $0<\beta\leq\infty$; see e.g.
\cite[Theorem 2.7.1]{T83}. In this section, we extend this Sobolev-type
embedding to mixed-norm Herz--Triebel--Lizorkin spaces.

To this end, we first establish the corresponding Sobolev embedding
properties for the sequence spaces
$\dot{E}_{\vec p}^{\vec\alpha,\vec r}f_{\beta}^{s}$. Let $0<\vec p,\vec q<\infty$, $\vec\alpha_1=(\alpha_1^{1},\ldots,\alpha_n^{1})\in\mathbb R^n$ and 
$\vec\alpha_2=(\alpha_1^{2},\ldots,\alpha_n^{2})\in\mathbb R^n$.
Recall that
\[
\frac{1}{\mathbf p}
=\sum_{i=1}^{n}\frac{1}{p_i},
\qquad
\frac{1}{\mathbf q}
=\sum_{i=1}^{n}\frac{1}{q_i},
\]
and
\[
\boldsymbol{\alpha}_{j}
=\sum_{i=1}^{n}\alpha_i^{\,j},
\qquad
j\in\{1,2\}.
\]

 For $v\in\{1,\ldots,n\}$, we define
\[
\frac{1}{\mathbf p_v}
=\sum_{i=v}^{n}\frac{1}{p_i},
\qquad
\frac{1}{\mathbf q_v}
=\sum_{i=v}^{n}\frac{1}{q_i},
\]
and  
\[
\boldsymbol{\tilde{\mathbf{\alpha}}_{j,v}}
=\sum_{i=v}^{n}\alpha_i^{\,j},
\qquad
j\in\{1,2\}.
\]
Furthermore, let
\[
s_2^{\,v}
=s_2-\frac{1}{\mathbf q_v}
+\frac{1}{\mathbf p_v}
+\boldsymbol{\tilde{\mathbf{\alpha}}_{1,v}}
-\boldsymbol{\tilde{\mathbf{\alpha}}_{2,v}}.
\]

The following theorem establishes a Sobolev-type embedding for the mixed-norm Herz--Triebel--Lizorkin spaces
$\dot{E}_{\vec p}^{\vec\alpha,\vec r}f_{\beta}^{s}$.
\begin{theorem}
	\label{embeddings-sobolev}
	Let $s_1,s_2\in\mathbb{R}$,
	$0<\vec q<\vec p<\infty$,
	$0<\vec r\le \infty$,
	$\vec\alpha_1=(\alpha_1^{1},\ldots,\alpha_n^{1})\in\mathbb{R}^n$,
	$\vec\alpha_2=(\alpha_1^{2},\ldots,\alpha_n^{2})\in\mathbb{R}^n$,
	and $0<\theta,\beta\le \infty$.
	Assume that
	\[
	\alpha_i^{2}\ge \alpha_i^{1}>-\frac1{p_i},
	\qquad i\in \{1,\dots,n\},
	\]
	and
	\begin{equation}
	s_1-\frac1{\mathbf p}-\boldsymbol{\alpha}_1
	=
	s_2-\frac1{\mathbf q}-\boldsymbol{\alpha}_2.
	\label{newexp1.1}
	\end{equation}
	Then
	\begin{equation}
	\dot{E}_{\vec q}^{\vec\alpha_2,\vec r}
	f_{\theta}^{s_2}
	\hookrightarrow
	\dot{E}_{\vec p}^{\vec\alpha_1,\vec r}
	f_{\beta}^{s_1}.
	\label{Sobolev-emb1.1}
	\end{equation}
\end{theorem}

\begin{proof}
By similarity, it suffices to consider the case $\beta=1$,
$1\leq \vec{p},\vec{q}<\infty$, and $1\leq \vec{r}\leq \infty$.
Let $\lambda \in \dot{E}_{\vec{q}}^{\vec{\alpha}_{2},\vec{r}}
f_{\theta}^{s_{2}}$.
The proof is divided into two steps.

\textit{Step 1.} Define
\[
\digamma
=
\Biggl\|
\sum_{v=0}^{\infty}
2^{v\left(s_{1}+\frac{n}{2}\right)}
\sum_{m\in\mathbb Z^{n}}
\lambda_{v,m}\chi_{v,m}
\Biggr\|_{\dot K_{p_{1}}^{\alpha_{1}^{1},r_{1}}}
\]
and 
\[
V
=
\sum_{v=0}^{\infty}2^{(s_{2}^{2}+\frac n2)v}
\sum_{\widetilde m\in\mathbb Z^{\,n-1}}
\Bigl\|
\sum_{m_1\in\mathbb Z}
\lambda_{v,m}\chi_{v,m_1}
\Bigr\|_{\dot K_{q_1}^{\alpha_1^{2},r_1}}
\chi_{v,\widetilde m},
\]
where
\[
m=(m_{1},\ldots,m_{n}),
\qquad
\tilde m=(m_{2},\ldots,m_{n}).
\]
We claim that 
\[
\digamma
\lesssim V.
\]
Furthermore, we claim that for every $j\in\{2,\ldots,n-1\}$

\begin{equation*}
\big\|\cdots \big\|V\big\|_{\dot{K}_{p_{2}}^{\alpha_{2}^{1},r_{2}}}
\cdots \big\|_{\dot{K}_{p_{j}}^{\alpha_{j}^{1},r_{j}}}
\end{equation*}
can be estimated from above by
\begin{equation}
c\sum_{v=0}^{\infty}
2^{(s_{2}^{j+1}+\frac{n}{2})v}
\sum_{\tilde m_{j}\in\mathbb Z^{\,n-j}}
\Big\|\cdots
\Big\|
\sum_{\bar m_{j}\in\mathbb Z^{j}}
\lambda_{v,m}\chi_{v,\bar m_{j}}
\Big\|_{\dot K_{q_{1}}^{\alpha_{1}^{2},r_{1}}}
\cdots
\Big\|_{\dot K_{q_{j}}^{\alpha_{j}^{2},r_{j}}}
\chi_{v,\tilde m_{j}},
\label{sobolev3}
\end{equation}
where
\[
m=(m_{1},\ldots,m_{n}),\qquad
\bar m_{j}=(m_{1},\ldots,m_{j}),
\qquad
\tilde m_{j}=(m_{j+1},\ldots,m_{n}).
\]

In \eqref{sobolev3}, we take $j=n-1$ and define
\begin{equation*}
\hat{\lambda}_{v,m_{n}}
=
\Big\|\cdots
\Big\|
\sum_{\bar m\in\mathbb Z^{\,n-1}}
\lambda_{v,m}\chi_{v,\bar m}
\Big\|_{\dot K_{q_{1}}^{\alpha_{1}^{2},r_{1}}}
\cdots
\Big\|_{\dot K_{q_{n-1}}^{\alpha_{n-1}^{2},r_{n-1}}},
\end{equation*}
where $\bar m=(m_{1},\ldots,m_{n-1})$.

Observe that
\begin{align*}
\|\lambda\|_{\dot E_{\vec p}^{\vec\alpha_{1},\vec r}f_{1}^{s_{1}}}
&\lesssim
\big\|\cdots \big\|V\big\|_{\dot K_{p_{2}}^{\alpha_{2}^{1},r_{2}}}
\cdots \big\|_{\dot K_{p_{n}}^{\alpha_{n}^{1},r_{n}}} \\
&\lesssim
\|\hat{\lambda}\|_{\dot K_{p_{n}}^{\alpha_{n}^{1},r_{n}}
	f_{1}^{s_{2}^{n}}}.
\end{align*}

Recall that
\[
\dot K_{q_{n}}^{\alpha_{n}^{2},r_{n}}
f_{\infty}^{s_{2}}
\hookrightarrow
\dot K_{p_{n}}^{\alpha_{n}^{1},r_{n}}
f_{1}^{s_{2}^{n}},
\]
see \cite{Drihem2.13}. Consequently,
\begin{equation*}
\|\lambda\|_{\dot E_{\vec p}^{\vec\alpha_{1},\vec r}f_{1}^{s_{1}}}
\lesssim
\|\hat{\lambda}\|_{\dot K_{q_{n}}^{\alpha_{n}^{2},r_{n}}
	f_{\infty}^{s_{2}}}
\lesssim
\,\|\lambda\|_{\dot E_{\vec q}^{\vec\alpha_{2},\vec r}
	f_{\infty}^{s_{2}}}.
\end{equation*}

This completes the proof of \eqref{Sobolev-emb1.1}.

\textit{Step 2.} We proceed to prove the assertions stated in Step 1.

\textit{Substep 2.1.} We show that $
\digamma
\lesssim V.$ By the definition of the Herz norm,
\begin{align}
\digamma
&\lesssim
\Biggl(
\sum_{k_{1}=-\infty}^{2}
2^{k_{1}\alpha_{1}^{1}r_{1}}
\Bigl\|
\sum_{v=0}^{\infty}
2^{v\left(s_{1}+\frac{n}{2}\right)}
\sum_{m\in\mathbb Z^{n}}
\lambda_{v,m}\chi_{v,m}\chi_{k_{1}}
\Bigr\|_{L^{p_{1}}(\mathbb R)}^{r_{1}}
\Biggr)^{1/r_{1}}
\label{est2-bis}
\\
&\quad+
\Biggl(
\sum_{k_{1}=3}^{\infty}
2^{k_{1}\alpha_{1}^{1}r_{1}}
\Bigl\|
\sum_{v=0}^{\infty}
2^{v\left(s_{1}+\frac{n}{2}\right)}
\sum_{m\in\mathbb Z^{n}}
\lambda_{v,m}\chi_{v,m}\chi_{k_{1}}
\Bigr\|_{L^{p_{1}}(\mathbb R)}^{r_{1}}
\Biggr)^{1/r_{1}}.
\label{est3-bis}
\end{align}

The term \eqref{est2-bis} is estimated as follows:
\begin{align*}
&
\Biggl(
\sum_{k_{1}=-\infty}^{2}
2^{k_{1}\alpha_{1}^{1}r_{1}}
\Bigl\|
\sum_{v=0}^{2-k_{1}}
2^{v\left(s_{1}+\frac{n}{2}\right)}
\sum_{m\in\mathbb Z^{n}}
\lambda_{v,m}\chi_{v,m}\chi_{k_{1}}
\Bigr\|_{L^{p_{1}}(\mathbb R)}^{r_{1}}
\Biggr)^{1/r_{1}}
\\
&\quad+
\Biggl(
\sum_{k_{1}=-\infty}^{2}
2^{k_{1}\alpha_{1}^{1}r_{1}}
\Bigl\|
\sum_{v=3-k_{1}}^{\infty}
2^{v\left(s_{1}+\frac{n}{2}\right)}
\sum_{m\in\mathbb Z^{n}}
\lambda_{v,m}\chi_{v,m}\chi_{k_{1}}
\Bigr\|_{L^{p_{1}}(\mathbb R)}^{r_{1}}
\Biggr)^{1/r_{1}}
\\
&=: I_{1}+I_{2}.
\end{align*}
\textit{Estimation of \(I_{1}\).}
Let \(x_{1}\in R_{k_{1}}\cap Q_{v,m_{1}}^{1}\) and
\(y_{1}\in Q_{v,m_{1}}^{1}\), where \(k_{1},m_{1}\in\mathbb Z\).
Since
\[
|x_{1}-y_{1}|\le 2^{1-v},
\]
it follows that
\[
|y_{1}|
\le |x_{1}-y_{1}|+|x_{1}|
\le 2^{1-v}+2^{k_{1}}
< 2^{3-v},
\]
because \(v\le 2-k_{1}\). Hence \(y_{1}\) belongs to the interval
\[
Q_{v}=(-2^{3-v},2^{3-v}).
\]

Therefore,
\[
|\lambda_{v,m}|^{t_{1}}
=2^{v}\int_{\mathbb R}
|\lambda_{v,m}|^{t_{1}}
\chi_{v,m_{1}}(y_{1})\,dy_{1}
\le
2^{v}\int_{Q_{v}}
|\lambda_{v,m}|^{t_{1}}
\chi_{v,m_{1}}(y_{1})\,dy_{1},
\]
whenever \(x_{1}\in R_{k_{1}}\cap Q_{v,m_{1}}^{1}\),
\(m=(m_{1},\ldots,m_{n})\in\mathbb Z^{n}\), and \(t_{1}>0\).

Consequently, for every \(x_{1}\in R_{k_{1}}\),
\[
\sum_{m_{1}\in\mathbb Z}
|\lambda_{v,m}|^{t_{1}}
\chi_{v,m_{1}}(x_{1})
\le
2^{v}
\int_{Q_{v}}
\sum_{m_{1}\in\mathbb Z}
|\lambda_{v,m}|^{t_{1}}
\chi_{v,m_{1}}(y_{1})\,dy_{1},
\]
and hence
\[
\sum_{m_{1}\in\mathbb Z}
|\lambda_{v,m}|^{t_{1}}
\chi_{v,m_{1}}(x_{1})
\le
2^{v}
\Bigl\|
\sum_{m_{1}\in\mathbb Z}
\lambda_{v,m}\chi_{v,m_{1}}\chi_{Q_v}
\Bigr\|_{L^{t_{1}}(\mathbb R)}^{t_{1}} .
\]

This yields
\begin{align*}
&
2^{\alpha_{1}^{1}k_{1}}
\Bigl\|
\sum_{v=0}^{2-k_{1}}
2^{v(s_{1}+\frac n2)}
\sum_{m\in\mathbb Z^{n}}
\lambda_{v,m}\chi_{v,m}\chi_{k_{1}}
\Bigr\|_{L^{p_{1}}(\mathbb R)}
\\
&\lesssim
2^{(\alpha_{1}^{1}+\frac1{p_{1}})k_{1}}
\sum_{v=0}^{2-k_{1}}
2^{v(s_{1}+\frac n2+\frac1{t_{1}})}
\sum_{\tilde m\in\mathbb Z^{n-1}}
\Bigl\|
\sum_{m_{1}\in\mathbb Z}
\lambda_{v,m}\chi_{v,m_{1}}\chi_{Q_v}
\Bigr\|_{L^{t_{1}}(\mathbb R)}
\chi_{v,\tilde m},
\end{align*}
where the implicit constant is independent of \(k_{1}\).

Choose \(t_{1}>0\) such that
\[
\frac1{t_{1}}
>
\max\Bigl(
\frac1{q_{1}},
\frac1{q_{1}}+\alpha_{1}^{2}
\Bigr).
\]
Using \eqref{newexp1.1} and Lemma~\ref{lem:lq-inequality}, since $\alpha_1^{1}>-\frac1{p_1}$, we obtain
\begin{align*}
I_{1}^{r_{1}}
&\lesssim
\sum_{v=0}^{\infty}
2^{v(
	s_{2}-\frac1{\mathbf q}
	-\boldsymbol{\alpha}_{2}
	+\frac1{\mathbf p_{2}}
	+\boldsymbol{\tilde{\mathbf{\alpha}}_{1,2}}
	+\frac1{t_{1}}
	+\frac n2
	)r_{1}}
\sum_{\tilde m\in\mathbb Z^{n-1}}
\Bigl\|
\sum_{m_{1}\in\mathbb Z}
\lambda_{v,m}\chi_{v,m_{1}}\chi_{Q_v}
\Bigr\|_{L^{t_{1}}(\mathbb R)}^{r_{1}}
\chi_{v,\tilde m}
\\
&\lesssim
\sum_{v=0}^{\infty}
2^{v(
	s_{2}-\frac1{\mathbf q}
	-\boldsymbol{\alpha}_{2}
	+\frac1{\mathbf p_{2}}
	+\boldsymbol{\tilde{\mathbf{\alpha}}_{1,2}}
	+\frac1{t_{1}}
	+\frac n2
	)r_{1}}
\sum_{\tilde m\in\mathbb Z^{n-1}}
\Bigl(
\sum_{h\le -v}
\Bigl\|
\sum_{m_{1}\in\mathbb Z}
\lambda_{v,m}
\chi_{v,m_{1}}
\chi_{h+3}
\Bigr\|_{L^{t_{1}}(\mathbb R)}^{\tau}
\Bigr)^{\frac {r_{1}} \tau}
\chi_{v,\tilde m},
\end{align*}
where $\chi_{h+3}$ denotes the characteristic function of the annulus
$R_{h+3}\subset \mathbb{R}$ and $0<\tau \leq \min(1,t_{1})$.

By H\"{o}lder's inequality, with
\[
\frac1d
=
\frac1{t_{1}}
-\frac1{q_{1}}
-\alpha_{1}^{2},
\]
we obtain
\begin{align*}
I_{1}^{r_{1}}
\lesssim
\sum_{v=0}^{\infty}
2^{v\frac{r_{1}}d}
\Bigl(
\sum_{h\le -v}
2^{h\tau(\frac1d+\alpha_{1}^{2})}
\sup_{j\ge0}\Bigl(
\sum_{\tilde m\in\mathbb Z^{n-1}}
\Bigl\|
2^{(s_{2}^{2}+\frac n2)j}
\sum_{m_{1}\in\mathbb Z}
\lambda_{j,m}\chi_{j,m_{1}}\chi_{h+3}
\Bigr\|_{L^{q_{1}}(\mathbb R)}^{\tau}
\chi_{j,\tilde m}\Bigl)
\Bigr)^{\frac {r_{1}} \tau} .
\end{align*}

Applying Lemma~\ref{lem:lq-inequality} once more yields
\begin{align*}
I_{1}
&\lesssim
\sum_{j=0}^{\infty}
2^{(s_{2}^{2}+\frac n2)j}
\sum_{\tilde m\in\mathbb Z^{n-1}}
\Bigl(
\sum_{h=0}^{\infty}
2^{-\alpha_{1}^{2}hr_{1}}
\Bigl\|
\sum_{m_{1}\in\mathbb Z}
\lambda_{j,m}\chi_{j,m_{1}}\chi_{-h+3}
\Bigr\|_{L^{q_{1}}(\mathbb R)}^{r_{1}}
\Bigr)^{1/r_{1}}
\chi_{j,\tilde m}
\\
&\lesssim
\sum_{j=0}^{\infty}
2^{(s_{2}^{2}+\frac n2)j}
\sum_{\tilde m\in\mathbb Z^{n-1}}
\Bigl\|
\sum_{m_{1}\in\mathbb Z}
\lambda_{j,m}\chi_{j,m_{1}}
\Bigr\|_{\dot K_{q_{1}}^{\alpha_{1}^{2},r_{1}}}
\chi_{j,\tilde m}.
\end{align*}
\textit{Estimation of $I_{2}$.}
We shall prove that
\begin{align*}
&2^{k_{1}\alpha_{1}^{1}}
\Big\|
\sum_{v=3-k_{1}}^{\infty}
2^{v(s_{1}+\frac n2)}
\sum_{m\in\mathbb Z^{n}}
\lambda_{v,m}\chi_{v,m}\chi_{k_{1}}
\Big\|_{L^{p_{1}}(\mathbb R)}
\\
&\lesssim
\sum_{j=0}^{\infty}
2^{j(s_{2}^{2}+\frac n2)}
\sum_{\tilde m\in\mathbb Z^{n-1}}
2^{k_{1}\alpha_{1}^{2}}
\Big\|
\sum_{m_{1}\in\mathbb Z}
\lambda_{j,m}\chi_{j,m_{1}}
\chi_{\widetilde R_{k_{1}}}
\Big\|_{L^{q_{1}}(\mathbb R)}
\chi_{j,\tilde m}
\\
&=:\delta ,
\end{align*}
for every $k_{1}\le 3$, where
\[
\widetilde R_{k_{1}}
=
\bigl\{
x_{1}\in\mathbb R:
2^{k_{1}-2}\leq|x_{1}|<2^{k_{1}+2}
\bigr\}.
\]

Equivalently, it suffices to show that
\begin{equation}
\int_{R_{k_{1}}}
2^{k_{1}\alpha_{1}^{1}p_{1}}
\delta^{-p_{1}}
\Big(
\sum_{v=3-k_{1}}^{\infty}
2^{v(s_{1}+\frac n2)}
\sum_{m\in\mathbb Z^{n}}
|\lambda_{v,m}|
\chi_{v,m}(x)
\Big)^{p_{1}}
dx_{1}
\lesssim 1.
\label{sobolevfor}
\end{equation}

The left-hand side of \eqref{sobolevfor} can be rewritten as
\[
\int_{R_{k_{1}}}
\Big(
\delta^{-1}
\sum_{v=3-k_{1}}^{\infty}
2^{v(\frac1{p_{1}}-\frac1{q_{1}})
	+(\alpha_{1}^{1}-\alpha_{1}^{2})(v+k_{1})
	+v(s_{2}^{2}+\frac n2)
	+k_{1}\alpha_{1}^{2}}
\sum_{m\in\mathbb Z^{n}}
|\lambda_{v,m}|
\chi_{v,m}(x)
\Big)^{p_{1}}
dx_{1}.
\]

Since $\alpha_{1}^{2}\ge \alpha_{1}^{1}$ and $v+k_{1}\ge3$,
\[
2^{(\alpha_{1}^{1}-\alpha_{1}^{2})(v+k_{1})}\le1.
\]
Hence  the left-hand side of \eqref{sobolevfor} is bounded by
\begin{align*}
&\int_{R_{k_{1}}}
\Big(
\delta^{-1}
\sum_{v=3-k_{1}}^{\infty}
2^{v(\frac1{p_{1}}-\frac1{q_{1}})
	+v(s_{2}^{2}+\frac n2)
	+k_{1}\alpha_{1}^{2}}
\sum_{m\in\mathbb Z^{n}}
|\lambda_{v,m}|
\chi_{v,m}(x)
\Big)^{p_{1}}
dx_{1}
\\
&=:T_{k_{1}}.
\end{align*}

We shall prove that \(T_{k_{1}}(\tilde x)\lesssim 1\) for every \(k_{1}\le 3\) and every $\tilde x=(x_{2},...,x_{n})\in \mathbb{R}^{n-1}$, where the implicit constant is independent of $k_{1}$ and $\tilde x$.
Our argument partially relies on the decomposition technique used in
\cite[Theorem~3.1]{V}; see also \cite{Drihem2.13}.

Set
\[
P_{k_{1}}
=
\sup_{v\ge 3-k_{1}}
2^{v(s_{2}^{2}+\frac n2)+k_{1}\alpha_{1}^{2}}
\sum_{m\in\mathbb Z^{n}}
|\lambda_{v,m}|\chi_{v,m}.
\]
Then
\[
T_{k_{1}}
=
\int_{R_{k_{1}}\cap\{P_{k_{1}}\le\delta\}}
\big(\cdots\,\big)^{p_{1}}dx_{1}
+
\int_{R_{k_{1}}\cap\{P_{k_{1}}>\delta\}}
\big(\cdots\,\big)^{p_{1}}dx_{1}
=:T_{k_{1}}^{1}+T_{k_{1}}^{2}.
\]

\medskip

\textit{Estimate of \(T_{k_{1}}^{1}\).}
Since $q_{1}<p_{1}$ and  \(P_{k_{1}}\le\delta\) on the domain of integration, we obtain
\[
T_{k_{1}}^{1}
\lesssim
\int_{R_{k_{1}}}
\left(
\delta^{-1}
P_{k_{1}}(x)
\right)^{q_{1}}
dx_{1}.
\]
Hence, by the definition of \(\delta\),
\[
T_{k_{1}}^{1}\lesssim 1,
\]
 where the implicit constant is independent of $k_{1}$ and $\tilde x$.
\medskip

\textit{Estimate of \(T_{k_{1}}^{2}\).}
Using a dyadic decomposition of the level set \(\{P_{k_{1}}>\delta\}\), we write
\[
T_{k_{1}}^{2}
=
\sum_{N=0}^{\infty}
\int_{R_{k_{1}}
	\cap
	\{2^{N/q_{1}}
	<
	\delta^{-1}P_{k_{1}}
	\le
	2^{(N+1)/q_{1}}\}}
\big(\cdots\,\big)^{p_{1}}dx_{1}.
\]

For a fixed \(N\), decompose
\[
\sum_{v=3-k_{1}}^{\infty}
2^{v(\frac 1{p_{1}}-\frac 1{q_{1}})
	+v(s_{2}^{2}+\frac n2)
	+k_{1}\alpha_{1}^{2}}
\sum_{m\in\mathbb Z^{n}}
|\lambda_{v,m}|\chi_{v,m}
=
J_{1,k_{1}}+J_{2,k_{1}},
\]
where
\[
J_{1,k_{1}}
=
\sum_{v=3-k_{1}}^{N}
2^{v(\frac 1{p_{1}}-\frac 1{q_{1}})
	+v(s_{2}^{2}+\frac n2)
	+k_{1}\alpha_{1}^{2}}
\sum_{m\in\mathbb Z^{n}}
|\lambda_{v,m}|\chi_{v,m},
\]
and
\[
J_{2,k_{1}}
=
\sum_{v=N+1}^{\infty}
2^{v(\frac 1{p_{1}}-\frac 1{q_{1}})
	+v(s_{2}^{2}+\frac n2)
	+k_{1}\alpha_{1}^{2}}
\sum_{m\in\mathbb Z^{n}}
|\lambda_{v,m}|\chi_{v,m}.
\]
Here we put $J_{1,k_{1}}=0$ if $k_{1}<3-N$. Let \(x_{1}\in R_{k_{1}}\cap Q_{v,m_{1}}^{1}\)  and
\(y_{1}\in Q_{v,m_{1}}^{1}\) with $3-k_{1}\leq v\leq N$.
Since
\[
|x_{1}-y_{1}|\le 2^{1-v},
\]
it follows that
\[
2^{k_{1}-2}
<
|y_{1}|
<
2^{1-v}+2^{k_{1}}
<
2^{k_{1}+2},
\]
and therefore \(y_{1}\in\widetilde R_{k_{1}}\). Consequently,
\[
|\lambda_{v,m}|^{q_{1}}
=
2^{v}
\int_{\mathbb R}
|\lambda_{v,m}|^{q_{1}}
\chi_{v,m_{1}}(y_{1})\,dy_{1}
\le
2^{v}
\int_{\widetilde R_{k_{1}}}
|\lambda_{v,m}|^{q_{1}}
\chi_{v,m_{1}}(y_{1})\,dy_{1}.
\]

Hence,
\begin{align*}
&\sum_{\widetilde m\in\mathbb Z^{n-1}}
\chi_{v,\widetilde m}
\sum_{m_{1}\in\mathbb Z}
|\lambda_{v,m}|^{q_{1}}
\chi_{v,m_{1}}
\\
&\le
2^{v}
\sum_{\widetilde m\in\mathbb Z^{n-1}}
\chi_{v,\widetilde m}
\int_{\widetilde R_{k_{1}}}
\sum_{m_{1}\in\mathbb Z}
|\lambda_{v,m}|^{q_{1}}
\chi_{v,m_{1}}(y_{1})\,dy_{1}
\\
&=
2^{v}
\sum_{\widetilde m\in\mathbb Z^{n-1}}
\Bigl\|
\sum_{m_{1}\in\mathbb Z}
\lambda_{v,m}
\chi_{v,m_{1}}
\chi_{\widetilde R_{k_{1}}}
\Bigr\|_{L^{q_{1}}(\mathbb R)}^{q_{1}}
\chi_{v,\widetilde m}
\\
&\le
2^{(\frac1{q_{1}}-s_{2}^{2}-\frac n2)q_{1}v
	-\alpha_{1}^{2}q_{1}k_{1}}
\delta^{q_{1}}.
\end{align*}

This leads to
\[
\delta^{-1}J_{1,k_{1}}
\lesssim
\sum_{v=3-k_{1}}^{N}
2^{v/p_{1}}
\lesssim
2^{N/p_{1}},
\]
where the implicit constant is independent of $\tilde x$.

For \(J_{2,k_{1}}\), we use the definition of \(P_{k_{1}}\):
\begin{align*}
\delta^{-1}J_{2,k_{1}}
&\lesssim
\sum_{v=N+1}^{\infty}
2^{v(\frac1{p_{1}}-\frac1{q_{1}})}
\bigl(
\delta^{-1}
P_{k_{1}}
\bigr)
\\
&\lesssim
\sum_{v=N+1}^{\infty}
2^{(N-v)/q_{1}}\,2^{v/p_{1}}
\lesssim
2^{N/p_{1}},
\end{align*}
whenever \(	2^{N/q_{1}}<\delta^{-1}P_{k_{1}}
\le2^{(N+1)/q_{1}}\)  and we have used  \(q_{1}<p_{1}\).

Therefore,
\[
\delta^{-1}(J_{1,k_{1}}+J_{2,k_{1}})
\lesssim
2^{N/p_{1}},
\]
whenever \(	2^{N/q_{1}}<\delta^{-1}P_{k_{1}}
\le2^{(N+1)/q_{1}}\). Combining the above estimates yields
\begin{align*}
T^{2}_{k_{1}}
&\lesssim
\sum_{N=0}^{\infty}
\int_{R_{k_{1}}
	\cap
	\{2^{N/q_{1}}
	<
	\delta^{-1}P_{k_{1}}
	\le
	2^{(N+1)/q_{1}}\}}
2^{N}\,dx_{1}
\\
&\lesssim
\int_{R_{k_{1}}}
\bigl(
\delta^{-1}
P_{k_{1}}(x)
\bigr)^{q_{1}}
dx_{1}
\\
&\lesssim 1,
\end{align*}
where the implicit constant is independent of $\tilde x$.

\textit{Estimate of \eqref{est3-bis}}. The arguments here are quite similar
to those used in the estimation of $I_{2}$.

\textit{Substep 2.2}. We shall prove that
\[
\|V\|_{\dot K_{p_2}^{\alpha_2^{1},r_2}}
\]
is bounded by
\begin{equation}
c
\sum_{v=0}^{\infty}
2^{(s_{2}^{3}+\frac n2)v}
\sum_{\bar m\in\mathbb Z^{\,n-2}}
\Bigl\|
\Bigl\|
\sum_{\check m\in\mathbb Z^{2}}
\lambda_{v,m}\chi_{v,\check m}
\Bigr\|_{\dot K_{q_1}^{\alpha_1^{2},r_1}}
\Bigr\|_{\dot K_{q_2}^{\alpha_2^{2},r_2}}
\chi_{v,\bar m},
\label{desired-est}
\end{equation}
where
\[
m=(m_1,\ldots,m_n),\qquad
\check m=(m_1,m_2),\qquad
\bar m=(m_3,\ldots,m_n).
\]

Define
\[
\widetilde{\lambda}_{v,\widetilde m}
=
\Bigl\|
\sum_{m_1\in\mathbb Z}
\lambda_{v,m}\chi_{v,m_1}
\Bigr\|_{\dot K_{q_1}^{\alpha_1^{2},r_1}},
\qquad
\widetilde m=(m_2,\ldots,m_n)\in\mathbb Z^{\,n-1}.
\]
Then
\[
\|V\|_{\dot K_{p_2}^{\alpha_2^{1},r_2}}
=
\Biggl(
\sum_{k_2\in\mathbb Z}
2^{k_2\alpha_2^{1}r_2}
\Bigl\|
\sum_{v=0}^{\infty}
2^{(s_{2}^{2}+\frac n2)v}
\sum_{\widetilde m\in\mathbb Z^{\,n-1}}
\widetilde{\lambda}_{v,\widetilde m}
\chi_{v,\widetilde m}\chi_{k_2}
\Bigr\|_{L^{p_2}(\mathbb R)}^{r_2}
\Biggr)^{1/r_2}.
\]

Splitting the outer sum at \(k_2=2\), we obtain
\[
\|V\|_{\dot K_{p_2}^{\alpha_2^{1},r_2}}
\lesssim I_3+I_4,
\]
where
\begin{align*}
I_3
&=
\Biggl(
\sum_{k_2=-\infty}^{2}
2^{k_2\alpha_2^{1}r_2}
\Bigl\|
\sum_{v=0}^{\infty}
2^{(s_{2}^{2}+\frac n2)v}
\sum_{\widetilde m\in\mathbb Z^{\,n-1}}
\widetilde{\lambda}_{v,\widetilde m}
\chi_{v,\widetilde m}\chi_{k_2}
\Bigr\|_{L^{p_2}(\mathbb R)}^{r_2}
\Biggr)^{1/r_2},\\
I_4
&=
\Biggl(
\sum_{k_2=3}^{\infty}
2^{k_2\alpha_2^{1}r_2}
\Bigl\|
\sum_{v=0}^{\infty}
2^{(s_{2}^{2}+\frac n2)v}
\sum_{\widetilde m\in\mathbb Z^{\,n-1}}
\widetilde{\lambda}_{v,\widetilde m}
\chi_{v,\widetilde m}\chi_{k_2}
\Bigr\|_{L^{p_2}(\mathbb R)}^{r_2}
\Biggr)^{1/r_2}.
\end{align*}

Since
\[
s_{2}^{2}
=
s_{2}^{3}
+\frac1{p_2}
-\frac1{q_2}
-\alpha_2^{2}
+\alpha_2^{1},
\]
the estimates for \(I_3\) and \(I_4\) follow exactly as in Substep~2.1.
Hence, \eqref{desired-est} holds.

Proceeding inductively, one obtains the assertion of Step~1 for every
\(j\in\{2,\ldots,n-1\}\). This completes the proof of Step~1.
\end{proof}

From Theorems \ref{phi-tran} and \ref{embeddings-sobolev}, we have the
following Sobolev embedding for\ spaces\ $\dot{E}_{\vec{p}}^{\mathbf{\vec{%
\alpha}},\vec{q}}F_{\beta }^{s}$.

\begin{theorem}
\label{embeddings3.1}Let $s_{1},s_{2}\in \mathbb{R},0<\vec{q}<\vec{p}<\infty
,0<\vec{r}\leq \infty ,\vec{\alpha}_{1}=(\alpha _{1}^{1},...,\alpha
_{n}^{1})\in \mathbb{R}^{n},\vec{\alpha}_{2}=(\alpha _{1}^{2},...,\alpha
_{n}^{2})\in \mathbb{R}^{n}$ and $0<\theta ,\beta \leq \infty $. \textit{We
suppose that}%
\begin{equation*}
\alpha _{i}^{2}\geq \alpha _{i}^{1}>-\frac{1}{p_{i}},\quad i\in \{1,...,n\}
\end{equation*}%
\textit{and }%
\begin{equation}
s_{1}-\frac{1}{\mathbf{p}}-\boldsymbol{\alpha}_{1}\leq s_{2}-\frac{1}{\mathbf{q}%
}-\boldsymbol{\alpha}_{2}.  \label{cond-sobolev}
\end{equation}%
Then%
\begin{equation*}
\dot{E}_{\vec{q}}^{\vec{\alpha}_{2},\vec{r}}F_{\theta
}^{s_{2}}\hookrightarrow \dot{E}_{\vec{p}}^{\vec{\alpha}_{1},\vec{r}%
}F_{\beta }^{s_{1}}.
\end{equation*}
\end{theorem}

\begin{remark}
 We prove the necessity of \eqref{cond-sobolev}. Let
	$\omega\in\mathcal{S}(\mathbb{R}^n)$ satisfy
	\[
	\operatorname{supp}\mathcal{F}{\omega}
	\subset
	\left\{\xi\in\mathbb{R}^n:\frac34<|\xi|<1\right\}.
	\]
	For $N\in\mathbb N$, define
	\[
	f_N(x)=\omega(2^N x), \qquad x\in\mathbb{R}^n.
	\]
	Clearly,
	\[
	\omega\in
	\dot{E}_{\vec{q}}^{\vec{\alpha}_2,\vec{\theta}}(\mathbb{R}^n)
	\cap
	\dot{E}_{\vec{p}}^{\vec{\alpha}_1,\vec{r}}(\mathbb{R}^n).
	\]
	By the support properties of $\omega$, we have
	\[
	\mathcal{F}^{-1}\varphi_j * f_N=
	\begin{cases}
	f_N, & j=N,\\
	0, & j\neq N.
	\end{cases}
	\]
	
	Let $i\in\{1,\ldots,n\}$. Then
	\begin{align*}
	\|f_N\|_{\dot K_{p_i}^{\alpha_i^1,r_i}}
	&=
	\left(
	\sum_{k_i=-\infty}^{\infty}
	2^{k_i\alpha_i^1 r_i}
	\|f_N\chi_{k_i}\|_{L^{{p_{i}}}(\mathbb R)}^{r_i}
	\right)^{1/r_i}
	\\
	&=
	2^{-N/p_i}
	\left(
	\sum_{k_i=-\infty}^{\infty}
	2^{k_i\alpha_i^1 r_i}
	\|\widetilde{\omega}(2^N\cdot)\chi_{k_i+N}\|_{L^{{p_{i}}}(\mathbb R)}^{r_i}
	\right)^{1/r_i}
	\\
	&=
	2^{-(\alpha_i^1+\frac1{p_i})N}
	\|\widetilde{\omega}(2^N\cdot)\|_{\dot K_{p_i}^{\alpha_i^1,r_i}},
	\end{align*}
	where
	\[
	\tilde{\omega}(2^{N}x)=\omega
	(2^{N}x_{1}...,2^{N}x_{i-1},x_{i},2^{N}x_{i+1},...,2^{N}x_{n}).
	\]
	
	Consequently,
	\[
	\|f_N\|_{\dot{E}_{\vec{p}}^{\vec{\alpha}_1,\vec r}F_\beta^{s_1}}
	=
	2^{s_1N}
	\|f_N\|_{\dot{E}_{\vec{p}}^{\vec{\alpha}_1,\vec r}(\mathbb R^n)}
	=
	2^{\left(
		s_1-\boldsymbol{\alpha}_1-\frac1{\mathbf p}
		\right)N}
	\|\omega\|_
	{\dot{E}_{\vec{p}}^{\vec{\alpha}_1,\vec r}(\mathbb R^n)}.
	\]
	
	Similarly,
	\[
	\|f_N\|_{\dot{E}_{\vec{q}}^{\vec{\alpha}_2,\vec{\theta}}F_\beta^{s_2}}
	=
	2^{\left(
		s_2-\boldsymbol{\alpha}_2-\frac1{\mathbf q}
		\right)N}
	\|\omega\|_
	{\dot{E}_{\vec{q}}^{\vec{\alpha}_2,\vec{\theta}}(\mathbb R^n)}.
	\]
	
	Assume that the embedding \eqref{Sobolev-emb1.1} holds. Then, for every
	$N\in\mathbb N$,
	\[
	\|f_N\|_{\dot{E}_{\vec{p}}^{\vec{\alpha}_1,\vec r}F_\beta^{s_1}}
	\le
	C\,
	\|f_N\|_{\dot{E}_{\vec{q}}^{\vec{\alpha}_2,\vec{\theta}}F_\beta^{s_2}},
	\]
	which yields
	\[
	2^{\left(
		s_1-s_2
		-\boldsymbol{\alpha}_1+\boldsymbol{\alpha}_2
		-\frac1{\mathbf p}
		+\frac1{\mathbf q}
		\right)N}
	\le C.
	\]
	Letting $N\to\infty$, we conclude that
	\[
	s_1-\boldsymbol{\alpha}_1-\frac1{\mathbf p}
	\le
	s_2-\boldsymbol{\alpha}_2-\frac1{\mathbf q},
	\]
	which is precisely condition \eqref{cond-sobolev}. 
	
	\end{remark}

From Theorem\ \ref{embeddings3.1}\ and the fact that $\dot{E}_{\vec{p}}^{\vec{0},%
\vec{p}}F_{\beta }^{s_{1}}=F_{\vec{p},\beta }^{s_{1}}$\ we immediately
arrive at the following results.

\begin{theorem}
\label{embeddings4.1}Let $s_{1},s_{2}\in \mathbb{R},\vec{\alpha}=(\alpha
_{1},...,\alpha _{n})\in \lbrack 0,\infty \mathbb{)}^{n},0<
\vec{r}\leq\vec{p}<\infty,0<\vec{q}
<\vec{p}<\infty ,0<\beta \leq \infty $\ and%
\begin{equation*}
s_{1}-\frac{1}{\mathbf{p}}\leq s_{2}-\frac{1}{\mathbf{q}}-\boldsymbol{\alpha}.
\end{equation*}%
Then%
\begin{equation*}
\dot{E}_{q}^{\vec{\alpha},\vec{r}}F_{\theta }^{s_{2}}\hookrightarrow F_{\vec{%
p},\beta }^{s_{1}}.
\end{equation*}
\end{theorem}

\begin{theorem}
\label{embeddings5.1}Let $s_{1},s_{2}\in \mathbb{R},0<\vec{q}\leq \vec{r}%
<\infty ,0<\vec{q}<\vec{p}<\infty $ and $0<\beta \leq \infty $. Assume that $%
\vec{\alpha}=(\alpha _{1},...,\alpha _{n})$\ with $\alpha _{i}>-\frac{1}{%
p_{i}},\ i\in \{1,...,n\}$ and 
\begin{equation*}
s_{1}-\frac{1}{\mathbf{p}}-\boldsymbol{\alpha}\leq s_{2}-\frac{1}{\mathbf{q}}.
\end{equation*}%
Then%
\begin{equation*}
F_{\vec{q},\theta }^{s_{2}}\hookrightarrow \dot{E}_{\vec{p}}^{\vec{\alpha},%
\vec{r}}F_{\beta }^{s_{1}}.
\end{equation*}
\end{theorem}

Using our results, we have the following useful consequences.

\begin{corollary}
Let $s_{1},s_{2},s_{3}\in \mathbb{R},0<\vec{t}<\vec{q}<\vec{p}<\infty
,0<\beta \leq \infty $\ and%
\begin{equation*}
s_{1}-\frac{1}{\mathbf{p}}=s_{2}-\frac{1}{\mathbf{q}}=s_{3}-\frac{1}{\mathbf{%
t}}.
\end{equation*}%
Then%
\begin{equation*}
F_{\vec{t},\infty }^{s_{3}}\hookrightarrow \dot{E}_{\vec{q}}^{\vec{0},\vec{p}%
}F_{\infty }^{s_{2}}\hookrightarrow F_{\vec{p},\beta }^{s_{1}}.
\end{equation*}
\end{corollary}

To prove this it is sufficient to choose in Theorem \ref{embeddings4.1}, $%
\vec{r}=\vec{p}$ and $\vec{\alpha}=\vec{0}$.\ However the desired embeddings are
an immediate consequence of the fact that 
\begin{equation*}
F_{\vec{t},\infty }^{s_{3}}\hookrightarrow F_{\vec{q},\infty }^{s_{2}}=\dot{E%
}_{\vec{q}}^{\vec{0},\vec{q}}F_{\infty }^{s_{2}}\hookrightarrow \dot{E}_{\vec{q}%
}^{\vec{0},\vec{p}}F_{\infty }^{s_{2}}.
\end{equation*}

\begin{corollary}
Let $s_{1},s_{2}\in \mathbb{R},0<\vec{q}<\vec{p}<\infty ,s_{1}-\frac{1}{%
\mathbf{p}}=s_{2}-\frac{1}{\mathbf{q}}$\ and $0<\beta \leq \infty $. Then%
\begin{equation*}
F_{\vec{q},\infty }^{s_{2}}\hookrightarrow \dot{E}_{\vec{p}}^{\vec{0},\vec{q}%
}F_{\beta }^{s_{1}}\hookrightarrow F_{\vec{p},\beta }^{s_{1}}.
\end{equation*}
\end{corollary}

\begin{proof}
To prove this it is sufficient to choose in Theorem\ \ref{embeddings5.1}, $%
\vec{r}=\vec{q}$ and $\vec{\alpha}=\vec{0}$. Then the desired embeddings are an
immediate consequence of the fact that 
\begin{equation*}
F_{\vec{q},\infty }^{s_{2}}\hookrightarrow \dot{E}_{\vec{p}}^{\vec{0},\vec{q}%
}F_{\beta }^{s_{1}}\hookrightarrow \dot{E}_{\vec{p}}^{\vec{0},\vec{p}}F_{\beta
}^{s_{1}}=F_{\vec{p},\beta }^{s_{1}}.
\end{equation*}
\end{proof}

\begin{remark}
Our Sobolev embeddings obtained in this section extend and improve the
corresponding results\ of\ \cite{JS07}.
\end{remark}

\textbf{Data Availability Statement} Data sharing is not applicable to this article as no
data sets were generated or analyzed during this study.

\section*{Declarations}

\textbf{Funding and/or Conflicts of interests/Competing interests:} This work was supported by the General Direction of Higher Education and Training under Grant No. C00L03UN280120220004 and by the General Directorate of Scientific Research and Technological Development, Algeria. The author declares that there are no conflicts of interest or competing interests

\end{document}